%% file: v3.tex
\documentclass{amsart}
\usepackage{helvet, color}
\usepackage{amscd,amsmath,amsxtra,amsthm,amssymb,stmaryrd,xr,mathrsfs,mathtools,enumerate,commath}
\usepackage[dvipsnames]{xcolor}
\definecolor{vegasgold}{rgb}{0.77, 0.7, 0.35}
\definecolor{darkgoldenrod}{rgb}{0.72, 0.53, 0.04}
\definecolor{gold(metallic)}{rgb}{0.83, 0.69, 0.22}
\usepackage{hyperref}
\hypersetup{
 colorlinks=true,
 linkcolor=Maroon,
 filecolor=brown,
 urlcolor=gold(metallic),
 citecolor=darkgoldenrod,
 pdftitle={Statistics of Anticyclotomic Invariants},
 }

\usepackage[all,cmtip]{xy}
\usepackage[margin=1.2 5in]{geometry}

\DeclareFontFamily{U}{wncy}{}
\DeclareFontShape{U}{wncy}{m}{n}{<->wncyr10}{}
\DeclareSymbolFont{mcy}{U}{wncy}{m}{n}
\DeclareMathSymbol{\Sh}{\mathord}{mcy}{"58}

\newtheorem{theorem}{Theorem}[section]
\newtheorem{lemma}[theorem]{Lemma}
\newtheorem{question}[theorem]{Question}
\newtheorem{conj}[theorem]{Conjecture}
\newtheorem{heuristic}[theorem]{Heuristic}
\newtheorem{proposition}[theorem]{Proposition}
\newtheorem{corollary}[theorem]{Corollary}
\newtheorem{definition}[theorem]{Definition}
\newtheorem{fact}[theorem]{Fact}
\numberwithin{equation}{section}

\theoremstyle{remark}
\newtheorem{remark}[theorem]{Remark}
\newtheorem{example}[theorem]{Example}

\newcommand{\tr}{\operatorname{tr}}
\newcommand{\Frob}{\operatorname{Frob}}
\newcommand{\Gal}{\operatorname{Gal}}

\newcommand{\ns}{\operatorname{ns}}
\newcommand{\Aut}{\operatorname{Aut}}

\newcommand{\cE}{\mathcal{E}}

\newcommand{\Qp}{\mathbb{Q}_p}
\newcommand{\Zp}{\mathbb{Z}_p}

\newcommand{\GL}{\mathrm{GL}}

\newcommand{\Z}{\mathbb{Z}}
\newcommand{\p}{\mathfrak{p}}
\newcommand{\Q}{\mathbb{Q}}
\newcommand{\F}{\mathbb{F}}
\newcommand{\C}{\mathbf{C}}

\newcommand{\cL}{\mathcal{L}}

\newcommand{\Hom}{\mathrm{Hom}}
\newcommand{\new}{\mathrm{new}}
\newcommand{\base}{\mathrm{base}}
\newcommand{\Sel}{\mathrm{Sel}}

\newcommand{\Gr}{\mathrm{Gr}}
\newcommand{\Reg}{\mathrm{Reg}}

\newcommand{\rank}{\mathrm{rank}}

\newcommand{\corank}{\mathrm{corank}}

\newcommand{\Kinf}{K_\infty}

\newcommand{\X}{\mathcal X}
\newcommand{\tor}{\mathrm{tors}}

\newcommand{\op}[1]{\operatorname{#1}}
\newcommand{\Selp}{\Sel_{p^{\infty}}(E/K_{\infty})}
\newcommand{\Selpell}{\Sel_{p^{\infty}}(E/K_{\infty}^{d})}
\newcommand{\mup}{\mu_p(E/K_{\infty})}
\newcommand{\lap}{\lambda_p(E/K_{\infty})}
\newcommand\mtx[4] { \left( {\begin{array}{cc}
 #1 & #2 \\
 #3 & #4 \\
 \end{array} } \right)}

 \newcommand{\widebar}[1]{\mkern 2.5mu\overline{\mkern-2.5mu#1\mkern-2.5mu}\mkern 2.5mu}

 \newcommand{\barrho}{\overline{\rho}}
 \newtheorem{hypothesis}[theorem]{Hypothesis}
 \newcommand{\Kellinf}{K_{\infty}^{d}}

\begin{document}
\title[Arithmetic statistics and anticyclotomic Iwasawa theory]{Statistics for anticyclotomic Iwasawa Invariants of Elliptic Curves}

\author[J.~Hatley]{Jeffrey Hatley}
\address[Hatley]{
Department of Mathematics\\
Union College\\
Bailey Hall 202\\
Schenectady, NY 12308\\
USA}
\email{hatleyj@union.edu}

\author[D.~Kundu]{Debanjana Kundu}
\address[Kundu]{MAGC 3.434 \\
Department of Mathematical and Statistical Sciences\\
UTRGV\\
Edinburg, TX 78541\\
USA}
\email{dkundu@math.toronto.edu}

\author[A.~Ray]{Anwesh Ray}
\address[Ray]{Chennai Mathematical Institute, H1, SIPCOT IT Park, Kelambakkam, Siruseri, Tamil Nadu 603103, India}
\email{anwesh@cmi.ac.in}

\begin{abstract}
We study the average behaviour of the Iwasawa invariants for Selmer groups of elliptic curves, considered over anticyclotomic $\Z_p$-extensions in both the definite and indefinite settings.
The results in this paper lie at the intersection of arithmetic statistics and Iwasawa theory.
\end{abstract}


\maketitle

\section{Introduction}
\label{section:intro}
In \cite{CL84}, H.~Cohen and H.~W.~Lenstra introduced a plausible heuristic on the distribution of class groups of number fields.
It was claimed (and justified) that their asymptotic behaviour mimics that of generic finite abelian groups weighted inversely by the size of the automorphism group.
Over the years, these heuristics have been suitably modified (see for example \cite{CM90, CM94, Mal08, Mal10, Gar15, AM15}) and are known to match numerical results fairly well.
In \cite{FW89}, E.~Friedman and L.~C.~Washington reinterpreted the Cohen--Lenstra heuristics using random matrix theory.
These ideas have since become a powerful source of predictions for number fields, which kick-started the field of arithmetic statistics.
One route that arithmetic statistics took was in the direction of elliptic curves.
Motivated by the strong analogy between number fields and elliptic curves, C.~Delaunay modelled Tate--Shafarevich groups of elliptic curves based on the Cohen--Lenstra heuristics, see \cite{Del01}.

An important conjecture in the theory of elliptic curves is the \emph{rank distribution conjecture} which claims that over any number field, half of all elliptic curves have Mordell--Weil rank zero and the remaining half have Mordell--Weil rank one.
Finally, higher Mordell--Weil ranks constitute zero percent of all elliptic curves, even though there may exist infinitely many such elliptic curves.
Therefore, a suitably-defined \emph{average rank} would be $1/2$.
The best results in this direction are by M.~Bharagava and A.~Shankar (see \cite{BS15_quartic, BS15_cubic}).
They show that the average rank of elliptic curves over $\Q$ is strictly less than one, and that both rank zero and rank one cases comprise non-zero densities across all elliptic curves over $\Q$.
Proving results about the Mordell--Weil rank almost always involves a thorough analysis of the Selmer group.
This idea dates back to the proof of the Mordell--Weil theorem, and is an essential part of results of Bhargava--Shankar.
In \cite{BS13_4Selmer, BS13_5Selmer}, they explicitly computed the average sizes of certain Selmer groups to deduce asymptotic results of Mordell--Weil ranks and made the following conjecture.
\begin{conj}
Let $F$ be a number field and $n$ be any positive integer.
Then when all elliptic curves $E$ are ordered by height,
the average size of the $n$-Selmer group, denoted by $\Sel_n(E/F)$, is $\sigma(n)$, the sum of the divisors of $n$.
\end{conj}
This conjecture has been verified for $n= 2,3,4,5$ and was enough
to deduce powerful partial results for the rank distribution conjecture.

Another subject of active research is Iwasawa theory.
Introduced by K.~Iwasawa in the 1950's, it started as the study of class groups over infinite towers of number fields (see \cite{Iwa59_GammaExtensions, Iwa59_cyclotomic}).
The close relationship between the group of units of number fields and the rational points of elliptic curves motivated B.~Mazur to study the Iwasawa theory of Selmer groups of elliptic curves in \cite{Maz72}.

In \cite{KR21}, the second and the third named authors initiated the study of average behaviour of the Iwasawa invariants for the $p$-primary Selmer groups of elliptic curves over the cyclotomic $\Z_p$-extension of $\Q$.
In this paper, we extend the ideas to study the average behaviour of the $p$-primary Selmer group of an elliptic curve with good \emph{ordinary} reduction at $p$ over the \emph{anticyclotomic} $\Z_p$-extension of an imaginary quadratic field $K/\Q$.
These two settings are substantially different from each other.
Over the cyclotomic $\Z_p$-extension of $\Q$, Mazur conjectured that for an elliptic curve with good ordinary reduction at an odd prime $p$, the $p$-primary Selmer group is cotorsion as a module over the Iwasawa algebra, denoted by $\Lambda$.
This conjecture was settled by K.~Kato, see \cite[Theorem 17.4]{Kat04}.
Unlike in the cyclotomic case, M.~Bertolini showed (see \cite{Ber95}) that over an anticyclotomic $\Z_p$-extension, the Selmer group need not always be $\Lambda$-cotorsion.
Several authors have studied classical Selmer groups of elliptic curves (more generally, modular forms) in the anticyclotomic setting (see for example \cite{Vat03, BD05, PW11}).

The Iwasawa algebra is isomorphic to the power series ring $\Z_p\llbracket T\rrbracket$.
The algebraic structure of the Selmer group (as a $\Lambda$-module) is encoded by Iwasawa invariants, $\mu$ and $\lambda$.
By the $p$-adic Weierstrass Preparation Theorem, the characteristic ideal of the Pontryagin dual of the Selmer group is generated by a unique element $f_E^{(p)}(T)$, which can be expressed as a power of $p$ times a distinguished polynomial.
The $\mu$-invariant is the power of $p$ dividing $f_E^{(p)}(T)$ and the $\lambda$-invariant is its degree.
In \cite[Theorem 5.11]{PW11}, R.~Pollack--T.~Weston show that when condition (CR) holds the Selmer group \emph{is} $\Lambda$-cotorsion and the algebraic $\mu$-invariant is 0 (see \cite{kim2017freeness} for a corrected version of condition (CR)).
In a slightly more general setting, vanishing of the $\mu$-invariant can be deduced from the arguments given in the last two pages of \cite{Castella}.
Further, if the $p$-primary part of the Tate--Shafarevich group, denoted by $\Sh(E/K)[p^\infty]$, is finite, then the $\lambda$-invariant is at least as large as the Mordell--Weil rank of $E$ over $K$ (see Lemma \ref{lemma21}).
However, the $\lambda$-invariant may indeed be strictly larger than the rank, and one of our main objectives is to determine its behaviour on average.
We analyze the following three separate but interrelated questions.
\begin{enumerate}
\item\label{question1} For a fixed elliptic curve $E$ over a fixed imaginary quadratic field $K$, how do the Iwasawa invariants vary as $p$ varies over all odd primes $p$ at which $E$ has good ordinary reduction?
\item\label{question2} For a fixed elliptic curve $E_{/\Q}$ and a fixed odd prime $p$ of good ordinary reduction, how do the Iwasawa invariants vary as $K = \Q(\sqrt{-d})$ varies over all primes $d>0$?
\item\label{question3} For a fixed prime $p$ and a fixed imaginary quadratic field $K$, how do the Iwasawa invariants vary as $E$ varies over all elliptic curves defined over $\Q$, ordered by height (with good ordinary reduction at $p$)?
\end{enumerate}

Let $N$ denote the conductor of $E$.
We write $N=N^+ N^-$, where $N^+$ (resp. $N^-$) is divisible only by primes which are split (resp. inert) in $K$.
When $N^-$ is divisible by an odd number of primes, we are said to be working in the \textit{definite case}, and when $N^+$ is divisible by an even number of primes, it is called the \textit{indefinite case}.
There is a sharp divide in anticyclotomic Iwasawa theory between these two cases due to the presence of Heegner points in the latter setting.

With regards to the first question, we first work in the \emph{definite case} and show that for non-CM elliptic curves, the exact order of growth for the number of primes at which $\mu=0$ is closely related to the Lang--Trotter Conjecture.
A more precise answer to Question 1 is provided in Theorem \ref{th46} and Corollary \ref{cor: to main thm of section 4}.
The second question, is subtle and we can only provide a partial result in this direction (see Theorem \ref{thm: partial result for varying K}).
It is shown that this question is largely dependent on how often the order of the Tate--Shafarevich group over $K$ is not divisible by $p$.
We address Question 3 in Theorem \ref{main result varying elliptic curve} for elliptic curves of Mordell--Weil rank $0$, and the answer is largely dependent on the variation of Tate--Shafarevich groups over all elliptic curves (of rank 0).
In this paper, we not only extend, but also refine the approach taken in \cite{KR21}, and relate the results to Cohen--Lenstra heuristics.

In Section \ref{sec:indefinite}, we study the aforementioned three questions in the \emph{indefinite case}, i.e., when the \emph{Heegner hypothesis} is satisfied and the Selmer group of interest in \emph{not} $\Lambda$-cotorsion.
Note that there is no known formula for the Euler characteristic of the $p$-primary Selmer group in the indefinite setting.
The approach used here is to relate the Selmer group, which is not cotorsion to one which is.
The characteristic element of this modified Selmer group is related to the Bertolini--Darmon--Prasanna $p$-adic $L$-function (see \cite{BDP}). We use this relationship to prove our results in this indefinite case.
It appears that answering questions \eqref{question1}-\eqref{question3} systematically in the indefinite setting is currently out of reach.
We provide some partial answers in Section \ref{stats for indefinite} and supplement our results with computational data.

\emph{Organization:} Including this introduction, this article has 11 sections.
Section \ref{section: background} is preliminary in nature.
We give definitions of the objects of interest and record some basic results from the literature.
In Section \ref{section: EC}, we discuss in detail the Euler characteristic associated to the Selmer group of an elliptic curve.
We prove several important lemmas which are used throughout this article.

In Sections \ref{section: fixed E varying p} through \ref{section: CL heuristics for Sha} we focus on the definite setting, where $K$ fails to satisfy the Heegner hypothesis.
In Section \ref{section: fixed E varying p}, we fix an imaginary quadratic field $K/\Q$ and consider the anticyclotomic $\Z_p$-extension $K_\infty/K$.
Using the Euler characteristic formula, we study the variation of the Iwasawa invariants of a fixed rank 0 elliptic curve (defined over $\Q$) without complex multiplication as $p$ varies.
In Section \ref{section: Ridgdill}, we give sufficient conditions for elliptic curves to have finitely many anomalous primes both over $\Q$ and over $K$.
In Section \ref{section: vary K vanish mu}, for a given pair $(E,p)$ we study for what proportion of imaginary quadratic fields the dual Selmer group is a finitely generated $\Z_p$-module over the anticyclotomic $\Z_p$-extension.
Further, we supplement our results with concrete examples.
In Section \ref{section: vary K rank 0}, we move one step ahead and study for what proportion of imaginary quadratic fields the dual Selmer group is trivial.
In Section \ref{section: vary elliptic curve}, we extend the results proved in \cite{KR21}.
We study the variation of Iwasawa invariants of the Selmer groups $\Sel_{p^{\infty}}(E/\Q_{\infty})$ and $\Sel_{p^{\infty}}(E/K_{\infty})$ as $E$ ranges over all elliptic curves (defined over $\Q$) of rank zero (over $K$) with good ordinary reduction at the prime(s) above $p$.
In Section \ref{section: CL heuristics for Sha}, we survey the heuristics for the Tate--Shafarevich groups for elliptic curves ordered by height and explain their consequences to our results.

In Section \ref{sec:indefinite}, we turn to the indefinite setting, where the Iwasawa theory is significantly different due to the presence of Heegner points.
In particular, the Selmer groups of interest are no longer $\Lambda$-cotorsion, so many of the strategies used in the preceding sections fail.
Finally, computational evidence for our results in the definite and indefinite settings are presented in tables in Section \ref{section: tables}.

\subsection*{Acknowledgements}
DK acknowledges the support of the PIMS Postdoctoral Fellowship.
We thank Ashay Burungale, Francesc Castella, Henri Darmon, Chan-Ho Kim, Barry Mazur, Ravi Ramakrishna, Lawrence Washington, Tom Weston, and Stanley Xiao for helpful discussions.
We thank Lawrence Washington for pointing out a mistake made in the calculations of Table 2 in an earlier version of the paper.
We thank the referee for their comments which helped improve the exposition.

\section{Background and Preliminaries}
\label{section: background}

\subsection{} Throughout, let $p$ be an odd prime and $E$ an elliptic curve defined over an imaginary quadratic field $K=\Q(\sqrt{-d})$.
Denote by $\Kinf$ the anticyclotomic $\Z_p$-extension of $K$.
This is the unique $\Z_p$-extension of $K$ which is Galois over $\Q$, for which the Galois group $\Gal(\Kinf/\Q)$ is pro-dihedral.
Set $\Gamma$ to be the Galois group $\Gal(\Kinf/K)$ and pick a topological generator $\gamma\in \Gamma$.
The Iwasawa algebra $\Lambda$ is the completed group algebra $\Z_p\llbracket \Gamma \rrbracket :=\varprojlim_n \Z_p[\Gamma/\Gamma^{p^n}]$.
Fix the isomorphism of rings $\Lambda\simeq\Z_p\llbracket T\rrbracket $, that sends $\gamma -1$ to the formal variable $T$.

Assume throughout that $E$ has good ordinary reduction at the primes $v|p$ of $K$.
We recall the definition of the $p$-primary Selmer group of $E$ over $\Kinf$, which will be the main object of study in the paper.
Let $S_p$ denote the primes $v$ of $K$ such that $v|p$.
Choose a finite set of primes $S$ containing $S_p$ and the primes at which $E$ has bad reduction.
For any finite extension $L/K$, write $S(L)$ for the set of primes $v$ of $K$ such that $v$ lies above a prime in $S$.
Write
\[
J_v(E/L) := \bigoplus_{w|v} H^1\left( L_w, E\right)[p^\infty]
\]
where the direct sum is over all primes $w$ of $L$ lying above $v$.
Then the \emph{$p$-primary Selmer group over $L$} is defined as follows
\[
\Sel_{p^\infty}(E/L):=\ker\left\{ H^1\left(K_S/L,E[p^{\infty}]\right)\longrightarrow \bigoplus_{v\in S} J_v(E/L)\right\}.
\]
Next, set $J_v(E/\Kinf)$ to be the direct limit \[J_v(E/\Kinf):=\varinjlim_L J_v(E/L),\]where $L$ ranges over all number fields contained in $\Kinf$.
Taking direct limits, the \emph{$p$-primary Selmer group over $\Kinf$} is defined as follows
\[
\Selp:=\ker\left\{ H^1\left(K_S/\Kinf,E[p^{\infty}]\right)\longrightarrow \bigoplus_{v\in S} J_v(E/\Kinf)\right\}.
\]
Note that $J_v(E/\Kinf)$ is a $\Lambda$-module and the map above is a map of $\Lambda$-modules.
A $\Lambda$-module $M$ is said to be cofinitely generated (resp. cotorsion) if its Pontryagin dual $M^{\vee}:=\Hom_{\Z_p}\left(M, \Q_p/\Z_p\right)$ is finitely generated (resp. torsion) as a $\Lambda$-module.
A standard application of Nakayama's lemma shows that $\Selp$ is cofinitely generated; however, the Selmer group need \emph{not} always be cotorsion.
The property of being $\Lambda$-cotorsion holds if some additional hypotheses are satisfied, which are discussed in \S\ref{section: fixed E varying p}.

\subsection{}
Let $M$ be a cofinitely generated $\Lambda$-module.
We introduce the Iwasawa invariants associated to $M$.
By the \emph{Structure Theorem for $\Lambda$-modules} (see \cite[Theorem 13.12]{washington1997}), $M^{\vee}$ is pseudo-isomorphic to a finite direct sum of cyclic $\Lambda$-modules, i.e., there is a map of $\Lambda$-modules
\[
M^{\vee}\longrightarrow \Lambda^r\oplus \left(\bigoplus_{i=1}^s \Lambda/(p^{\mu_i})\right)\oplus \left(\bigoplus_{j=1}^t \Lambda/(f_j(T)) \right)
\]
with finite kernel and cokernel.
Here, $r$ is the rank of $M$, $\mu_i>0$ and $f_j(T)$ is a distinguished polynomial (i.e., a monic polynomial with non-leading coefficients divisible by $p$).
The characteristic ideal of $M^\vee$ is (up to a unit) generated by
\[
f_{M}^{(p)}(T) = f_{M}(T) := p^{\sum_{i} \mu_i} \prod_j f_j(T).
\]
The $\mu$-invariant of $M$ is defined as the power of $p$ in $f_{M}(T)$.
More precisely,
\[
\mu_p(M):=\begin{cases}
\sum_{i=1}^s \mu_i & \textrm{ if } s>0\\
0 & \textrm{ if } s=0.
\end{cases}
\]
The $\lambda$-invariant of $M$ is the degree of the characteristic element, i.e.
\[
\lambda_p(M) :=\begin{cases}
\sum_{j=1}^t \deg f_j & \textrm{ if } t>0\\
0 & \textrm{ if } t=0.
\end{cases}
\]
Denote by $\mup$ (resp. $\lap$) the $\mu$-invariant (resp. $\lambda$-invariant) of the Selmer group $\Selp$.

\begin{lemma}\label{lemma21}
Let $E_{/K}$ be an elliptic curve and assume that the Selmer group $\Selp$ is cotorsion as a $\Lambda$-module.
Then $\lap\geq \rank_{\Z} E(K)$.
\end{lemma}

\begin{proof}
Set $r$ to denote the $\Z_p$-corank of $\Selp^{\Gamma}$.
Since $\Selp$ is cotorsion, $r$ is finite.
Consider the short exact sequence,
\[0\rightarrow E(K)\otimes \Q_p/\Z_p\rightarrow \Sel_{p^{\infty}}(E/K)\rightarrow \Sh(E/K)[p^{\infty}]\rightarrow 0.\]
We deduce that
\begin{equation}
\label{eq21}
 \corank_{\Z_p} \Sel_{p^{\infty}}(E/K)\geq \rank_{\Z} E(K),
\end{equation}
with equality if $\Sh(E/K)[p^{\infty}]$ is finite.
We know that $\lap\geq r$ (see also for example, \cite[Theorem~1.9]{Gre99} which can be adapted in this setting).
It suffices to show that $r\geq \rank_{\Z} E(K)$.
This is indeed the case, since there is a natural map
\[
\Sel_{p^{\infty}}(E/K)\rightarrow \Selp^\Gamma
\]with finite kernel.
From \eqref{eq21}, we see that $r\geq \rank_{\Z} E(K)$ and the result follows.
\end{proof}

\section{The Euler characteristic}
\label{section: EC}

In this section, we introduce the Euler characteristic associated to $\Selp$ and its relationship with the Iwasawa invariants $\mu$ and $\lambda$.

\subsection{} In what follows, $M$ will be a cofinitely generated and cotorsion $\Lambda$-module.
Since $\Gamma$ has $p$-cohomological dimension equal to $1$, the cohomology groups $H^i(\Gamma, M)$ are trivial for $i\geq 2$.
Recall that $H^1(\Gamma, M)$ is identified with the module of coinvariants $M_{\Gamma}$, see \cite[Proposition 1.7.7]{NSW08}.
Since $M$ is cofinitely generated as a $\Lambda$-module, both $M^{\Gamma}$ and $M_{\Gamma}$ are cofinitely generated as $\Z_p$-modules.
In fact, their coranks are equal, as the following result shows.

\begin{lemma}
\label{balancedrank}
Let $M$ be a cofinitely generated cotorsion $\Lambda$-module.
Then
\[
\corank_{\Z_p} H^0(\Gamma, M)=\corank_{\Z_p} H^1(\Gamma, M).
\]
\end{lemma}

\begin{proof}
See \cite[Lemma 2.1]{KR21}.
\end{proof}

As a result, $H^1(\Gamma, M)$ is finite if and only if $H^0(\Gamma, M)$ is finite.

\begin{definition}
Assume that the cohomology groups $H^i(\Gamma, M)$ are finite.
The \emph{Euler characteristic} of $M$ is defined as follows
\[\chi(\Gamma, M):=\frac{\# H^0(\Gamma, M)}{\# H^1(\Gamma, M)}.\]
\end{definition}
Denote by $f_{M}(T)$ be the characteristic element of the finitely generated torsion module $M^{\vee}$.
This is the unique generator of the characteristic ideal of $M^{\vee}$ which is a product of a power of $p$ and a \emph{distinguished polynomial}, i.e., $f_{M}(T)=p^{\mu}\times g_{M}(T)$, where $\mu\geq 0$ and $g_{M}(T)$ is a monic polynomial whose non-leading coefficients are divisible by $p$.
Note that $\lambda=\lambda(M)$ is the degree of $f_{M}(T)$.
Write
\[
f_{M}(T)=c_0 +c_1 T+\dots+c_{\lambda} T^{\lambda}
\]
and let $r_{M}\geq 0$ be the order of vanishing of $f_{M}(T)$ at $T=0$, i.e., the smallest index $j$ such that $c_j\neq 0$.
The order of vanishing $r_{M}$ is related to the Mordell--Weil rank of the elliptic curve $E$, as the following lemma shows.

\begin{lemma}\label{lemma23}
Let $E$ be an elliptic curve defined over an imaginary quadratic field $K$.
Assume that the following conditions are satisfied
\begin{enumerate}
 \item $M:=\Selp$ is $\Lambda$-cotorsion,
 \item $E(K)[p]=0$,
 \item $\Sh(E/K)[p^{\infty}]$ is finite.
\end{enumerate}
Then the order of vanishing $r_{M}$ is equal to $\rank_{\Z} E(K)$.
Further, the cohomology groups $H^i(\Gamma, M)$ are finite if and only if $E(K)$ is finite.
\end{lemma}

\begin{proof}
It follows from \cite[Lemma 2.11]{zerbes09} that
\begin{equation}\label{e1}r_{M}= \corank_{\Z_p} M^{\Gamma}.\end{equation}
Since we have assumed that $E(K)[p]=0$, the natural map
\[
\iota:\Sel_{p^{\infty}}(E/K)\rightarrow \Selp^{\Gamma}
\]
is injective. It follows from the Control Theorem (see \cite[Theorem 1]{Gre03}) that $\iota$ has finite cokernel.
Therefore, $r_{M}$ is equal to the corank of $\Sel_{p^\infty}(E/K)$.
From the short exact sequence
\[
0\rightarrow E(K)\otimes \Q_p/\Z_p\rightarrow \Sel_{p^\infty}(E/K)\rightarrow \Sh(E/K)[p^{\infty}]\rightarrow 0
\]
and our assumption on the finiteness of the Tate--Shafarevich group, we deduce that $r_{M}=\rank_{\Z} E(K)$.
In particular, from \eqref{e1} it follows that $M^{\Gamma}$ is finite if and only if $E(K)$ is finite.
By Lemma \ref{balancedrank}, $H^1(\Gamma, M)$ is finite precisely when $M^{\Gamma}$ is finite.
This completes the proof of the lemma.
\end{proof}

Let $E_{/K}$ be a rank 0 elliptic curve satisfying the conditions of Lemma \ref{lemma23}.
Then the Euler characteristic of $M=\Selp$ can be defined.
In this case, set $\chi(\Gamma, E[p^{\infty}]):=\chi(\Gamma, M)$.
Since $r_{M}= 0$, it follows that the constant term $c_0$ of $f_{M}(T)$ is non-zero.
As the following result shows, $c_0$ is closely related to $\chi(\Gamma, E[p^{\infty}])$.
Let $a, b\in \Q_p$, we say that $a\sim b$ if $a=ub$ for a unit $u\in \Z_p^{\times}$.

\begin{lemma}\label{lemma24}
Let $M$ be a cofinitely generated and cotorsion $\Lambda$-module for which the cohomology groups $H^i(\Gamma, M)$ are finite.
Then
\[
\chi(\Gamma, M)\sim c_0,
\]
where $c_0$ is the constant term of the characteristic element $f_{M}(T)$.
\end{lemma}

\begin{proof}
It follows from \cite[Lemma 2.13]{zerbes09} that if $M'$ is another $\Lambda$-module which is pseudo-isomorphic to $M$ then the Euler characteristics match, i.e.
\[
\chi(\Gamma, M)=\chi(\Gamma, M').
\]
Denote by $X$ the Pontryagin dual $M^{\vee}$ and assume without loss of generality that $X$ is isomorphic to a direct sum of cyclic $\Lambda$-modules as follows
\[
X\xrightarrow{\sim} \left(\bigoplus_{i=1}^s \Lambda/(p^{\mu_i})\right)\oplus \left(\bigoplus_{j=1}^t \Lambda/(f_j(T)) \right).
\]
Note that for $i=0$, $1$ we have that
\[
\# H^i(\Gamma, M)=\# H^{1-i}(\Gamma, X).
\]
Therefore, we arrive at the following relations
\[
\#H^0(\Gamma, M)=\# X/T X=\prod_{i=1}^s p^{\mu_i}\times \prod_{j=1}^t \abs{f_j(0)}_p^{-1}=\abs{f_{M}(0)}_p^{-1},
\]
and
\[
\#H^1(\Gamma, M)=\# \ker \left( X\xrightarrow{\times T} X \right)=1.
\]
Combining the above equalities, we find that
\[
\chi(\Gamma, M)=\frac{\# H^0(\Gamma, M)}{\# H^1(\Gamma, M)}=\abs{f_{M}(0)}_p^{-1}\sim c_0.
\]
\end{proof}

\begin{remark}
The Euler characteristic is a priori possibly a negative power of $p$.
However, the above relation $\chi(\Gamma, M)\sim c_0$ shows that it is indeed an integer given by a non-negative power of $p$.
\end{remark}

\subsection{} We discuss the relationship between $\chi(\Gamma, M)$ and the vanishing of Iwasawa invariants.

\begin{proposition}\label{prop36}
Let $M$ be a cofinitely generated and cotorsion $\Lambda$-module for which the cohomology groups $H^i(\Gamma, M)$ are finite.
Then the following are equivalent
\begin{enumerate}
 \item\label{261} $\chi(\Gamma, M)=1$,
 \item\label{262} $f_{M}(T)=1$,
 \item\label{263} $\mu(M)=0$ and $\lambda(M)=0$,
 \item\label{264} $M$ has finite cardinality.
\end{enumerate}
\end{proposition}

\begin{proof}
By Lemma \ref{lemma24}, the constant coefficient $c_0$ of $f_{M}(T)$ is a unit in $\Z_p$ if and only if $\chi(\Gamma, M)=1$.
Recall that $f_{M}(T)$ is the unique generator of the characteristic ideal of $M^{\vee}$ and is expressed as a power of $p$ times a distinguished polynomial.
Therefore, $c_0$ is a unit if and only if $f_{M}(T)=1$.
This proves that \eqref{261} and \eqref{262} are equivalent.

Recall that $\mu(M)$ is the power of $p$ dividing $f_{M}(T)$ and $\lambda(M)=\deg f_{M}(T)$.
Therefore, if $\mu(M)=0$ and $\lambda(M)=0$, then $f_{M}(T)\in \Z_p^{\times}$.
However, it must also be a distinguished polynomial.
Hence, $f_{M}(T)=1$.
Conversely, it is clear that if $f_{M}(T)=1$, then $\mu(M)=0$ and $\lambda(M)=0$.
As a result, \eqref{262} is equivalent to \eqref{263}.

It remains to show that \eqref{262} is equivalent to \eqref{264}.
If $f_{M}(T)=1$, then $M$ is pseudo isomorphic to $0$.
Hence, $M$ is finite.
On the other hand, if $M$ is finite, then any generator of the characteristic ideal of $M$ must be a unit in $\Lambda$.
Since $f_{M}(T)$ is defined to be a product of a power of $p$ and a distinguished polynomial, it follows that $f_{M}(T)=1$.
Therefore, \eqref{262} is equivalent to \eqref{264}.
\end{proof}

\section{Results for a fixed elliptic curve \texorpdfstring{$E$}{} and varying prime \texorpdfstring{$p$}{}}
\label{section: fixed E varying p}
Let $K=\Q(\sqrt{-d})$ be an imaginary quadratic field and $E$ a fixed elliptic curve over $\Q$ without complex multiplication.
Let $p$ be a prime at which $E$ has good ordinary reduction.
Let $S_p$ be the set of primes of $K$ above $p$.
The Iwasawa invariants $\mup$ and $\lap$ are associated to the Galois module $E[p^{\infty}]$ and depend on the prime $p$.
In this section, we study the variation of the Iwasawa invariants $\mup$ and $\lap$ as $p$ varies.

\subsection{}
First, we record certain well-known facts and conjectures.
Let $\widebar{\Q}$ be a choice of algebraic closure of $\Q$.
For any integer $N\geq 1$, define $E[N]:=\ker \left( E(\widebar{\Q})\xrightarrow{\times N} E(\widebar{\Q}) \right)$.
If $N'|N$, then consider the multiplication by $N/N'$-map
\[\iota_{N,N'}: E[N]\rightarrow E[N'].\]
The big Tate-module is defined as follows \[\op{T}(E):=\varprojlim_{N} E(\widebar{\Q})[N],\] where the inverse limit is taken with respect to the maps $\iota_{N, N'}$, whenever $N'|N$.
It is easy to see that $\op{T}(E)$ is non-canonically isomorphic to $\widehat{\Z}^2$, and admits a natural action of the absolute Galois group $\Gal(\widebar{\Q}/\Q)$.
Identify the group of $\widehat{\Z}$-linear automorphisms of $\op{T}(E)$ with $\GL_2(\widehat{\Z})$ and denote by
\[
\rho_{E}:\Gal(\widebar{\Q}/\Q)\rightarrow \GL_2(\widehat{\Z})
\]
the associated Galois representation.
Serre's open image theorem (see \cite[\S4 Theorem 3]{Serre72}) asserts that the image of $\rho_{E}$ has finite index in $\GL_2(\widehat{\Z})$.
Consider the residual representation \[
\barrho_{E,p}:\Gal(\widebar{\Q}/\Q)\rightarrow \GL_2(\Z/p\Z)\]
induced by $\Gal(\widebar{\Q}/\Q)$ on $E[p]$.
This residual representation may be reducible, i.e., contained in a conjugate of the Borel subgroup of upper triangular matrices $\mtx{\ast}{\ast}{0}{\ast}$ in $\GL_2(\Z/p\Z)$.
This holds for instance if $E(\Q)[p]\neq 0$.
However, since $[\GL_2(\hat{\Z}):\op{im}\rho_E]$ is finite, the residual representation $\barrho_{E,p}$ is surjective for all but finitely many primes $p$.

Recall that the Selmer group $\Selp$ was defined in \S\ref{section: background}.
Our first line of inquiry concerns conditions under which $\Selp$ is cotorsion as a $\Lambda$-module.
Consider the special case when $E$ is defined over $\Q$.
Let $N$ be the conductor of $E$.
Write $N=N^+ N^-$, where $N^+$ (resp. $N^-$) is divisible by primes that split (resp. are inert) in $K$.
If $N^-$ has an even number of prime divisors, then it is known by results of C.~Cornut \cite{Cor02} and V.~Vatsal \cite{Vat02} that $\Selp^{\vee}$ could have positive rank over $\Lambda$.
This is called the \textit{indefinite} setting, and we will study it in Section \ref{sec:indefinite}.

When $N^-$ has an odd number of prime divisors, the situation is more analogous to the cyclotomic Iwasawa theory of $E$.
This is called the \textit{definite} setting, and it will be our primary concern for the next six sections of the paper.
Let $\widetilde{E}$ denote the reduced curve modulo $p$ and set $a_p:=1+p-\#\widetilde{E}(\F_p)$.

\begin{hypothesis}
\label{ref hyp CR}
We refer to the following set of conditions as (CR).
\begin{enumerate}
\label{CR}
 \item $\barrho_{E,p}:\Gal(\widebar{\Q}/\Q)\rightarrow \GL_2(\F_p)$ is surjective.
 \item If $q$ is a prime with $q|N^-$ and $q\equiv \pm{1} \pmod{p}$, then $\barrho_{E,p}$ is ramified at $q$.
 \item $N^-$ is square-free and the number of primes dividing $N^-$ is odd.
 \item $a_p\not\equiv \pm 1\pmod{p}$.
\end{enumerate}
\end{hypothesis}
Then it follows from work of Vatsal \cite{Vat02}, M.~Bertolini--H.~Darmon \cite{BD05}, and Pollack--Weston \cite{PW11} that under (CR), the dual Selmer group $\Selp^{\vee}$ is a torsion module over the Iwasawa algebra $\Lambda$ with $\mup=0$.
Note that the last assumption $a_p\not\equiv \pm 1\pmod{p}$ needs to be added to the hypotheses of Pollack and Weston, according to \cite[Remark 1.4]{kim2017freeness}.
On the other hand, if $\Sel_{p^{\infty}}(E/K)$ is finite (i.e., the $\rank_{\Z} E(K)=0$ and $\Sh(E/K)[p^{\infty}]$ is finite), then it follows from \cite[Corollary 4.9]{Gre98_PCMS} that $\Selp^{\vee}$ is a torsion $\Lambda$-module.
However, under this hypothesis alone, it is not known that $\mup=0$.
Assume therefore, that $E$ is not CM, it is defined over the rationals, and that $N^-$ is square-free and divisible by an odd number of primes.
Note that for $p$ suitably large, the (CR) conditions are all satisfied provided $E$ has ordinary reduction at $p$ and $a_p\not\equiv \pm 1\pmod{p}$.
Therefore, their results give the following.

\begin{theorem}[Vatsal, Bertolini--Darmon, Pollack--Weston]
Let $E_{/\Q}$ be an elliptic curve without complex multiplication.
Assume that $N^-$ is square-free and divisible by an odd number of primes.
Consider the set of primes $\mathcal{P}_E$ for which
\begin{enumerate}
\item $E$ has good ordinary reduction at $p\in \mathcal{P}_E$,
 \item $a_p\not\equiv \pm 1\pmod{p}$.
 \end{enumerate}
Then for all but finitely many primes $p\in \mathcal{P}_E$, the Selmer group $\Selp^{\vee}$ is a $\Lambda$-torsion module and $\mup=0$.
\end{theorem}

It follows from the Hasse bound that a prime $p\geq 7$ is contained in $ \mathcal{P}_E$ if and only if $a_p\notin \{0, \pm 1\}$.
We recall the conjecture of Lang and Trotter.

\begin{conj}[Lang--Trotter]
\label{Lang Trotter}
Let $E_{/\Q}$ be an elliptic curve without complex multiplication and let $t$ be any integer.
For $x>0$,
let $\pi_{E,t}(x)$ be the number of primes $p\leq x$ such that $a_p(E)=t$.
Then there exists a constant $C_{E,t}\geq 0$ such that
\[\pi_{E,t}(x)\sim C_{E,t}\frac{\sqrt{x}}{\log x}.
\]
\end{conj}
\noindent Denote by $\mathcal{P}_E^c(x)$ the set of primes $p\leq x$ such that $p\notin \mathcal{P}_E$.
The Lang--Trotter conjecture predicts that
\[
\mathcal{P}_E^c(x)\sim C\frac{\sqrt{x}}{\log x},
\]
where, $C=C_{E,-1}+C_{E,0}+C_{E,1}$.
Thus, the conjecture gives an exact order of growth for the number of primes at which $\mu=0$ is not known.

\subsection{} Next, we describe the variation of the $\lambda$-invariant, by invoking the Euler characteristic formula.
Assume that $E_{/\Q}$ is an elliptic curve such that the Mordell--Weil group $E(K)$ and the Tate--Shafarevich group $\Sh(E/K)$ are finite.
Equivalently, $\Sel_{p^{\infty}}(E/K)$ is finite and $\Selp$ is cotorsion as a $\Lambda$-module.
Further, assume that $N^-$ is square-free and divisible by an odd number of primes.
By Lemma $\ref{lemma23}$, the cohomology groups $H^i(\Gamma, \Selp)$ are finite and the Euler characteristic $\chi(\Gamma, E[p^{\infty}]):=\chi(\Gamma, \Selp)$ is defined.
By Proposition $\ref{prop36}$,
\[
\mup=0\text{ and }\lap=0\Leftrightarrow \chi(\Gamma, E[p^{\infty}])=1.
\]

There is a formula for the Euler characteristic which may be interpreted as a $p$-adic analogue of the Birch and Swinnerton-Dyer formula.
The following discussion applies to \emph{any} elliptic curve $E$ defined over $K$ with good reduction at $p$.
Denote by $c_v=c_v(E)=[E(K_v):E_0(K_v)]$ the local Tamagawa factor at a finite prime $v\nmid p$ and $c_v^{(p)}:=|c_v|_p^{-1}$ be its $p$-part.
Set $\omega$ to be an invariant differential on $E$ and $\omega_v^*$ the N\'eron differential at $v$.
Denote by $ \Reg(E/K)$ the regulator of $E$ over $K$, i.e., the determinant of the canonical height pairing on $E(K)/E(K)_{\tor}$.
Set $D_K$ to denote the absolute discriminant of $K$ and $r:=\rank_{\Z} E(K)$.
The Birch and Swinnerton-Dyer conjecture predicts that the order of vanishing of the Hasse--Weil $L$-function $L(E/K,s)$ at $s=1$ is $r$.
Further, it implies that $\Sh(E/K)$ is finite, and on expanding $L(E/K,s)$ at $s=1$,
\[L(E/K,s)=a_r(s-1)^r+a_{r+1}(s-1)^{r+1}+\dots \]
where the leading coefficient $a_r$ is given by
\begin{equation}\label{eq:bsd}
a_r=\frac{\# \Sh(E/K) \times  \Reg(E/K)\times \Phi(E/K)}{\left(\# E(K)_{\tor}\right)^2 \times (D_K)^{\frac{1}{2}}}\times 2\int_{E(\mathbb{C})}\omega\wedge \bar{\omega}.
\end{equation}
Here, $\Phi(E/K)$ denotes the product
\[\Phi(E/K):=\prod_{v\nmid \infty} c_v \abs{\frac{\omega}{\omega_v^*}}_v.\]
As Lemma $\ref{lemma24}$ shows, when $r=0$, the Euler characteristic is related to the constant term of the characteristic element.
Thus, the Iwasawa main conjecture predicts that it is related to the constant term of the $p$-adic L-function.
Let $k_\p$ be the residue field of $K$ at a prime $\p|p$ and $\widetilde{E}$ be the reduction of $E$ at $\p$.
The following result is proved by extending the method in \cite[chapter 3]{CS00_GCEC}, in which the Euler characteristic formula for the cyclotomic $\Z_p$-extension is proved.

\begin{theorem}\label{thm:van-ord}
Let $E$ be an elliptic curve with good ordinary reduction at each prime above $p$ in $K$.
Assume that the following conditions are satisfied
\begin{enumerate}
 \item the Mordell--Weil group $E(K)$ is finite,
 \item $\Sh(E/K)[p^{\infty}]$ is finite.
\end{enumerate}
Then the Euler characteristic $\chi(\Gamma, E[p^{\infty}])$ is well-defined, and given by the formula
\begin{equation}\label{ecf}
\chi(\Gamma, E[p^{\infty}])\sim \frac{\# \Sh(E/K)[p^{\infty}]\times \left(\prod_{\p|p}\# \widetilde{E}(\kappa_\p)\right)^2}{\left(\# E(K)[p^\infty]\right)^2}\cdot \prod_{v\in S_{\ns}\setminus S_p} c_v^{(p)}(E).
\end{equation}
The set $S_{\ns}$ is the set of primes which are finitely decomposed in $K_{\infty}/K$.
\end{theorem}

\begin{proof}
The result is due to J.~van Order \cite[Theorem 1.1]{van15}, except that there is a slight inaccuracy in the Euler characteristic formula she provides.
To prove Theorem 1.1 (in \emph{loc. cit.}), the author requires Corollary 5.3 (in \emph{loc. cit}) which invokes an argument from \cite[Lemma 3.3]{Coa99}.
However, we note that this argument applies \emph{only} for the $\GL_2$ extension and not for the anticyclotomic $\Z_p$-extension.
In particular, when $v$ is a prime of $K$ not above $p$ which is finitely decomposed in the anticyclotomic $\Z_p$-extension, the map
\[
\gamma_v : J_v(E/K) \rightarrow J_v(E/K_{\infty,v})^{\Gamma}
\]
is surjective but not necessarily the zero map.
Using the argument in \cite[Lemma 3.4]{CS00_GCEC}, one can show that $\ker \gamma_v$ is finite of order $c_v^{(p)}$.
The same argument goes through because this computation is purely local and for an $\ell$-adic field (with $\ell\neq p$), the cyclotomic and the anticyclotomic $\Z_p$-extensions are the same when the latter is non-trivial.
In particular, both give the maximal unramified $\Z_p$ extension.
Therefore, each prime $v\in S_{\ns}\setminus S_p$ contributes a factor of $c_v^{(p)}(E)$ (instead of $c_v^{(p)}(E) \abs{L_v(E,1)}_p$).
The remainder of the proof of the (anticyclotomic) Euler characteristic formula can also be argued as in \cite[Chapter 3]{CS00_GCEC} or \cite{van15}.
\end{proof}

We remark in passing that when $\rank_{\Z} E(K)$ is positive, there is no known formula that generalizes \eqref{ecf} for the anticyclotomic $\Z_p$-extension.
However, for the cyclotomic $\Z_p$-extension, the formula which applies in the higher rank setting is due to Perrin-Riou and Schneider, see \cite{schneider1, schneider2}.
Note that in \eqref{ecf}, $\#E(K)[p^{\infty}]=1$ when $\barrho_{E,p}:\Gal(\widebar{K}/K)\rightarrow \GL_2(\Z/p\Z)$ is irreducible, which is the case for all but a finite number of prime numbers $p$ when $E$ does not have complex multiplication.
At this point, we introduce some simplifying notation.
Let $p$ be a prime and $K$ an imaginary quadratic field.
Consider the various quantities in \eqref{ecf} and set
\begin{itemize}
 \item $\Sh_p=\Sh_p(E/K):=\#\Sh(E/K)[p^{\infty}]$.
 \item For $\p|p$, set $\alpha_\p=\alpha_\p(E/K):=\#\widetilde{E}(k_\p)[p^{\infty}]$, and $\alpha_p=\alpha_p(E/K):=\prod_{\p|p} \alpha_\p$.
 \item For a non-archimedean prime $v\nmid p$, set $\tau_v=\tau_v(E/K)=c_v^{(p)}:=\abs{c_v}_p^{-1}$, and for $\ell\neq p$, write $\tau_\ell=\tau_\ell(E/K):=\prod_{v|\ell} \tau_v$.
\end{itemize}
The numbers $\Sh_p$, $\alpha_p$, and $\tau_\ell$ are all of the form $p^N$ for some $N\in \Z_{\geq 0}$.
When $\#E(K)[p]=1$, \eqref{ecf} simply becomes
\begin{equation}
 \chi(\Gamma, E[p^{\infty}])=\alpha_p^2\times \left(\Sh_p\times \prod_{v\in S_{\ns}\setminus S_p}\tau_v \right).
\end{equation}
Recall that we have assumed the finiteness of the Tate--Shafarevich group $\Sh(E/K)$, i.e., $\Sh_p=1$ for all but finitely many primes.
Furthermore, for any prime $v\nmid p$, $c_v^{(p)}=1$ for all but finitely many primes $p$.
However, it is possible for $\alpha_p\neq 1$ for an infinite number of primes $p$.
Such primes $p$ for which $\alpha_p\neq 1$ are known as \emph{anomalous}.
The following definition makes this notion precise.

\begin{definition}
Let $E_{/\Q}$ be an elliptic curve.
A prime $p$ at which $E$ has good reduction and $\widetilde{E}(\F_p)[p]\neq 0$ is said to be a \emph{$\Q$-anomalous} prime.
For $K$ any number field and $E_{/K}$ an elliptic curve with good (ordinary) reduction at $p$, the prime $p$ is called \emph{$K$-anomalous}\footnote{It is standard terminology to refer to the prime as being anomalous. However, we emphasize that the property of being anomalous depends on the elliptic curve as well.} if $\widetilde{E}(k_v)[p]\neq 0$ for some prime $v|p$ (with residue field $k_v$).
\end{definition}

\begin{lemma}\label{lemma46}
Suppose that $K$ is an imaginary quadratic field and $E_{/\Q}$ is an elliptic curve and $p\geq 7$ be a prime.
Set $a_p=1+p-\#\widetilde{E}(\F_p)$, there are two cases to consider.
\begin{enumerate}
 \item If $p$ is ramified or split in $K$, then $E$ is $K$-anomalous if and only if $a_p=1$.
 \item If $p$ is inert in $K$, then $E$ is $K$-anomalous if and only if $a_p=\pm 1$.
\end{enumerate}
\end{lemma}

\begin{proof}
We consider the two cases separately.

\emph{Case 1:} First consider the case when $p$ is ramified or split in $K$.
In this case, the residue field $k_v=\F_p$ for each of the primes $v|p$ of $K$.
By the Hasse bound, $E$ is $\Q$-anomalous if and only if $\#\widetilde{E}(\F_p)=p$, i.e., $a_p=1$.

\emph{Case 2:} Next, suppose that $p$ is inert in $K$.
Then there is one prime $v|p$ above $p$ with $k_v=\F_{p^2}$.
Let $\alpha, \beta$ be the roots of $X^2-a_p X+p$.
It follows from \cite[Chapter 5, Theorem 2.3.1]{silverman2009} that
\[
\# \widetilde{E}(\F_{p^2})=p^2+1-(\alpha^2+\beta^2)=(p+1)^2-a_p^2=(p+1-a_p)(p+1+a_p).
\]
According to the Hasse bound, $p$ divides $\# \widetilde{E}(\F_{p^2})$ if and only if $a_p=\pm 1$.
\end{proof}

Let $\Pi_E(x)$ (resp. $\Pi_{E,K}(x)$) be the number of primes $p\leq x$ that are $\Q$-anomalous (resp. $K$-anomalous).
The Lang--Trotter conjecture predicts that there are constants $0\leq c\leq C$ such that for large enough values of $x$,
\[\begin{split}
& \Pi_E(x)\sim c\frac{\sqrt{x}}{\log x} \quad \textrm{and}\\
& c\frac{\sqrt{x}}{\log x}\leq \Pi_{E,K}(x)\leq C\frac{\sqrt{x}}{\log x}.
\end{split}\]
In this context, the best known unconditional bound follows from a result of V.~K.~Murty (see \cite{Mur97})
\[
\Pi_E(x)\ll \frac{x (\log \log x)^2}{(\log x)^2}.
\]
It is indeed possible for the constants $c,C$ to be zero.
In fact, there are curves for which the number of $\Q$-anomalous primes is finite.
Such curves are studied in greater detail in the next section.


\begin{theorem}
\label{th46}
Let $E_{/\Q}$ be an elliptic curve without complex multiplication and $K$ be an imaginary quadratic field.
Assume that $\rank_{\Z} E(K)=0$ and $\Sh(E/K)$ is finite.
Let $\mathfrak{M}$ be the finite set of primes $p$ such that
\begin{enumerate}
\item $E$ has good ordinary reduction at $p$,
\item $E(K)[p]\neq 0$,
\item $\Sh_{p}\times \prod_{v\in S_{\ns}\setminus S_{p}}c_v^{(p)} \neq 1.$
\end{enumerate}
Let $p\notin \mathfrak{M}$ be a prime of good ordinary reduction of $E$.
Then the following assertions hold.
\begin{enumerate}
\item If $p$ is $K$-anomalous, then $\mup>0$ or $\lap>0$ (or both).
\item If $p$ is \emph{not} $K$-anomalous, then $\Selp=0$.
\end{enumerate}
\end{theorem}

\begin{proof}
Assume throughout that $E$ has good ordinary reduction at the prime $p$.
Since the Selmer group $\Sel_{p^{\infty}}(E/K)$ is finite, it follows that $\Selp^{\vee}$ is a torsion $\Lambda$-module, see \cite[Corollary 4.9]{Gre98_PCMS}.
Furthermore, the Euler characteristic $\chi(\Gamma, E[p^{\infty}])$ is well defined and is given by the formula
\[
\chi(\Gamma, E[p^{\infty}])\sim \frac{\alpha_p^2\times \left(\Sh_p\times \prod_{v\in S_{\ns}\setminus S_p}c_v^{(p)} \right)}{\left( \# E(K)[p^\infty]\right)^2}.
\]
For $p\notin \mathfrak{M}$, we have that
\[
\chi(\Gamma, E[p^{\infty}])=\alpha_p^2.
\]
Thus, the Euler characteristic $\chi(\Gamma, E[p^{\infty}])=1$ if and only if $p$ is not $K$-anomalous.

If $p$ is $K$-anomalous, then Proposition \ref{prop36} asserts that either $\mup>0$ or $\lap>0$ (or both).
On the other hand, if $p$ is not $K$-anomalous, then $\mup=0$ and $\lap=0$.
This implies that $\Selp$ is finite.
The result of Greenberg \cite[Proposition 4.14]{Gre99} for the cyclotomic $\Z_p$-extension generalizes verbatim to our setting, to show that $\Selp$ has no proper finite index submodules.
Therefore, if it is finite, then it necessarily follows that $\Selp=0$.
\end{proof}

\begin{corollary}
\label{cor: to main thm of section 4}
Let $E_{/\Q}$ be an elliptic curve and $K$ be an imaginary quadratic field such that the conditions of Theorem \ref{th46} are satisfied.
Let $x>0$ and $\Omega_E(x)$ be the number of primes $p\leq x$ satisfying the following conditions.
\begin{enumerate}
 \item $E$ has good ordinary reduction at $p$,
 \item $p$ splits in $K$,
 \item $\Selp\neq 0$.
\end{enumerate}
Then
\[
\Omega_E(x)\ll \frac{x (\log \log x)^2}{(\log x)^2}.
\]
\end{corollary}

\begin{proof}
Theorem \ref{th46} asserts that if $p$ is a prime for which the above conditions are satisfied, then $\alpha_p\neq 1$, i.e., $p$ is $K$-anomalous.
Since $p$ splits in $K$, we have that $\alpha_p\neq 1$ if and only if $\#\widetilde{E}(\F_p)[p]\neq 0$, i.e., $p$ is $\Q$-anomalous.
The result follows from the bound in \cite{Mur97} on the number of $\Q$-anomalous primes $p\leq x$.
\end{proof}

\section{Ridgdill curves}
\label{section: Ridgdill}
The notion of a finitely anomalous elliptic curve is due to P.~C.~Ridgdill, see \cite{Rid10}.

\begin{definition}
Let $E_{/\Q}$ be an elliptic curve with finitely many $\Q$-anomalous primes $p$, then we say $E$ is \emph{$\Q$-Ridgdill}.
Let $K$ be an imaginary quadratic field.
If $E$ has finitely many $K$-anomalous primes $p$, we say that $E$ is \emph{$K$-Ridgdill}.
\end{definition}
Note that by Lemma \ref{lemma46}, a $K$-Ridgdill curve is necessarily $\Q$-Ridgdill. The following is a corollary to Theorem \ref{th46}.

\begin{corollary}
Let $E_{/\Q}$ be an elliptic curve and $K$ be an imaginary quadratic field so that the conditions of Theorem \ref{th46} are satisfied.
Assume that $E$ is $\Q$-Ridgdill.
Then there are only finitely many primes $p$ which split in $K$ such that
\begin{enumerate}
\item $E$ has good ordinary reduction at $p$,
\item $\Selp\neq 0$.
\end{enumerate}
Furthermore, if $E$ is $K$-Ridgdill, then there are only finitely many primes $p$ (split and inert in $K$) such that the above conditions are satisfied.
\end{corollary}

Let $N_E$ denote the conductor of $E$ and $N$ be some fixed positive integer.
Ridgdill observes that if $E$ is an elliptic curve such that the image of
\[\barrho_{E,N}:\Gal(\widebar{\Q}/\Q)\rightarrow \GL_2(\Z/N \Z)\] is contained in a subgroup $H$ of $\GL_2(\Z/N\Z)$ containing no element whose trace is $1$, then all primes $p\nmid N_E N$ are not anomalous.
Hence, in this case, the elliptic curve $E$ is finitely anomalous.
The subgroups which do not contain any such element with trace $1$ are classified, and elliptic curves whose mod-$N$ representation is contained in these subgroups are classified.

It was first observed in \cite[Lemma 8.18]{Maz72} that there are many examples of $\Q$-Ridgdill curves.
We provide a (different) proof for the sake of completeness.

\begin{proposition}\label{proposition53}
Let $E$ be an elliptic curve over $\Q$ such that $E(\Q)_{\tor}\neq 0$.
Then $E$ is $\Q$-Ridgdill.
\end{proposition}
\begin{proof}
Let $\ell$ be a prime such that $E(\Q)[\ell]\neq 0$ and let $N_E$ be the conductor of $E$.
Denote by $\chi_\ell$ the mod-$\ell$ cyclotomic character.
Since $E(\Q)[\ell]\neq 0$, the mod-$\ell$ representation is of the form $\barrho_{E,\ell}=\mtx{1}{\ast}{0}{\chi_{\ell}}$.
Let ${\Frob}_p$ be the Frobenius at a prime $p\nmid \ell N_E$.
We claim that $p$ is not $\Q$-anomalous.
Indeed, we find that
\[
1+p-\#\widetilde{E}(\F_p)\equiv \tr\left(\rho_E({\Frob}_p)\right)\equiv 1+\chi(p)=1+p\pmod{\ell}.
\]
Therefore, $\ell$ divides $\# \widetilde{E}(\F_p)$.
Assume in addition to $p\nmid \ell N_E$ that $p\geq 7$.
It follows from the Hasse bound that $p| \#\widetilde{E}(\F_p)$ if and only if $\#\widetilde{E}(\F_p)=p$.
Since $\ell | \# \widetilde{E}(\F_p)$, we see that $p$ is not $K$-anomalous.
The only primes that can possibly be $\Q$-anomalous are the primes dividing $\ell N_E$.
Hence, $E$ is $\Q$-Ridgdill.
\end{proof}

In \cite[Lemma A.5]{MR05}, B.~Mazur--K.~Rubin provided sufficient conditions for an elliptic curve to be $K$-Ridgdill.
Here, we construct a different class of examples of $K$-Ridgdill curves.

\begin{proposition}
Let $E_{/\Q}$ be an elliptic curve with conductor $N_E$ and suppose that $K=\Q(\sqrt{-d})$ is an imaginary quadratic field.
Assume that there is a prime $\ell$ such that $E(\Q)[\ell]\neq 0$.
Then $E$ is $K$-Ridgdill in the following situations
\begin{enumerate}
 \item $\ell=2$,
 \item $\ell=3$ and $K=\Q(\sqrt{-3})$.
\end{enumerate}
Furthermore, if either of the above conditions are satisfied and $p\geq 7$ is a prime coprime to $N_E d \ell$, then $p$ is not $K$-anomalous.
\end{proposition}

\begin{proof}
Let $p\geq 7$ be a prime which is coprime to $N_E d\ell$ and assume that $p$ is $K$-anomalous.
Since $p$ is unramified in $K$, there are only two cases to consider.
Recall that $p$ is split in $K$ if and only if $\left(\frac{-d}{p}\right)=1$.

\emph{Case 1}: First, consider the case when $p$ is split in $K$.
Let $v|p$ be a prime of $K$.
In this case, the residue field $k_v$ of $K$ at $v$ is $\F_p$.
Let $\chi$ denote the $\ell$-adic cyclotomic character and $\chi_{\ell}$ be its mod-$\ell$ reduction.
Since $E(\Q)[\ell]\neq 0$, the argument from the proof of Proposition \ref{proposition53} shows that $\ell$ divides $\# \widetilde{E}(\F_p)$.
Next, it follows from the Hasse bound that $p| \#\widetilde{E}(\F_p)$ if and only if $\#\widetilde{E}(\F_p)=p$.
This contradicts our earlier conclusion that $\ell | \# \widetilde{E}(\F_p)$.
Hence, $p$ is not $K$-anomalous.

\emph{Case 2:} Next, suppose $p$ is inert in $K$. The same arguments (as in \emph{Case 1}) show that $\ell$ divides $\#\widetilde{E}(\F_p)=1+p-a_p$. Furthermore, since $p$ is $K$-anomalous, by Lemma~\ref{lemma46}(2) we have $a_p=\pm 1$. However,
since $\ell \neq p$, it once again follows from the Hasse bound $p\nmid \#\widetilde{E}(\F_p)$, and so we must have $a_p=-1$.
Therefore $p\equiv -2\pmod{\ell}$.

Suppose that $\ell=2$.
Since $p$ is not even, it is not $K$-anomalous.
Therefore, $E$ is $K$-Ridgdill if $E(\Q)[2]\neq 0$.
Next consider the case when $\ell=3$ and $K=\Q(\sqrt{-3})$.
We have that
\[\left(\frac{-3}{p}\right)=\left(\frac{-1}{p}\right)\left(\frac{3}{p}\right)=(-1)^{\frac{p-1}{2}}(-1)^{\frac{p-1}{2}\frac{3-1}{2}}\left(\frac{p}{3}\right)=\left(\frac{p}{3}\right)=\left(\frac{-2}{3}\right)=\left(\frac{1}{3}\right)=1.\]
However, since $p$ is assumed to be inert in $K$, it is not $K$-anomalous.
Therefore, $E$ is $K$-Ridgdill.
\end{proof}


\section{Results for varying imaginary quadratic number field \texorpdfstring{$K$}{}: Vanishing of the \texorpdfstring{$\mu$}{}-invariant}
\label{section: vary K vanish mu}

In this section, we fix an elliptic curve $E$ defined over $\mathbb{\Q}$ with conductor $N_E$ and an odd prime number $p$ at which $E$ has good ordinary reduction.
Let $d\neq p$ be a prime number and $K^{d}:=\Q(\sqrt{-d})$.
Denote by $\Kellinf$ the anticyclotomic $\Z_p$-extension of $K^{d}$ and recall that $\Selpell$ is the Selmer group over $\Kellinf$.

\begin{definition}
Let $\mathcal{S}$ be a subset of prime numbers.
Define the \emph{density of $\mathcal{S}$} as
\[
\delta(\mathcal{S}):=\lim_{x\rightarrow \infty} \frac{\#\{\ell\leq x| \ell\in \mathcal{S}\}}{\pi(x)},
\]
where $\pi(x)$ is the prime counting function.
\end{definition}
\noindent As $d$ varies over all prime numbers we study the following question
\begin{question}What is the density of primes $d$ such that $\Selpell^{\vee}$ is a torsion $\Lambda$-module with $\mu$-invariant equal to $0$?
\end{question}
We do not make any assumption on the rank of $E(\Q)$ and prove estimates for how often $\Selpell^{\vee}$ satisfies the conditions (CR) introduced in Hypothesis~\ref{ref hyp CR}.
For simplicity we will assume that $N_E$ is odd throughout the remainder of this section.

\begin{theorem}\label{th62}
Let $E_{/\Q}$ be a fixed elliptic curve.
Let $p$ be an odd prime number at which $E$ has good ordinary reduction and set $a_p:=1+p-\#\widetilde{E}(\F_p)$.
Additionally, assume that the following conditions are satisfied.
\begin{enumerate}[(i)]
\item The residual representation
\[
\barrho_{E,p}:\Gal(\widebar{\Q}/\Q)\rightarrow \GL_2(\F_p)
\]
is surjective.
\item There is a prime $q$ dividing the conductor $N_E$ of $E$ such that at least one of the following conditions is satisfied
\begin{enumerate}
\item $q\not \equiv \pm 1\pmod{p}$,
\item $\bar{\rho}_{E,p}$ is ramified at $q$.
\end{enumerate}
\item The conductor $N_E$ of $E$ is square-free,
\item $a_p\not \equiv \pm 1\pmod{p}$.
\end{enumerate}
Let $\mathcal{S}$ be the set of primes $d$ such that the conditions (CR) are satisfied for $(E, p, K^d)$.
Let $k$ be the number of primes $q| N_E$ such that both of the following conditions hold
\begin{enumerate}
 \item $q \equiv \pm 1\pmod{p}$,
 \item $\bar{\rho}_{E,p}$ is unramified at $q$.
\end{enumerate}
Then the density $\delta(\mathcal{S})=\frac{1}{2^{k+1}}$.
In particular, the proportion of primes $d$ for which $\Selpell$ is cotorsion as a $\Lambda$-module with $\mu$-invariant equal to zero is $\frac{1}{2^{k+1}}$.
\end{theorem}
Before commencing with the proof of the above result, we introduce some notation.
Let $r_1, \ldots, r_t$ be distinct prime numbers dividing $N_{E}$.
Denote this set of primes by $\Pi(N_E)$.
Here, $t=\omega(N_E)$ is the number of distinct prime factors of $N_E$.
Enumerate this set as follows: for $i\leq k$, let $r_i$ be the primes for which both the conditions are satisfied
\begin{enumerate}
\item $r_i \equiv \pm 1\pmod{p}$,
\item $\bar{\rho}_{E,p}$ is unramified at $r_i$.
\end{enumerate}
Assumption \emph{(ii)} implies that $k<\omega(N_E)$.
Let $d\nmid N_E$ be a prime number.
Write $N_E=N^+_{E,d} N^-_{E,d}$, such that $N^+_{E,d}$ (resp. $N^-_{E,d}$) is the product of primes that split (resp. are inert) in $K^{d}=\Q(\sqrt{-d})$.

Let $\Omega$ be any subset of $\Pi(N_E)$ and define $\pi_{\Omega}$ to be the set of primes $d\nmid N_E$ such that the primes in $\Omega$ (resp. $\Pi(N_E)\setminus \Omega$) are inert (resp. split) in $K^{d}$.
In other words, $\pi_{\Omega}$ consists of primes $d$ such that the primes $\Pi(N_E)$ are unramified in $K^d$ and $N_{E,d}^-=\prod_{q\in\Omega} q$.
For $x>0$, set $\pi_{\Omega}(x):=\#\{d\leq x| d\in \pi_{\Omega}\}$.
Denote by $\delta_{\Omega}$ the density of $\pi_{\Omega}$, defined as follows
\[
\delta_{\Omega}:=\delta(\pi_{\Omega})=\lim_{x\rightarrow \infty} \frac{\pi_{\Omega}(x)}{\pi(x)}.
\]
The next result shows that the set of primes $\pi_{\Omega}$ is determined by residue classes modulo $N_E$.
Let $(\Z/N_E \Z)^*$ denote the ring of units in $\Z/N_E \Z$.
Recall that $\varphi(N_E)=\prod_i (r_i-1)$ is the cardinality of $(\Z/N_E \Z)^*$.

\begin{lemma}
\label{residueclasses}
There is a collection of residue classes $\mathfrak{r}_{\Omega}\subset (\Z/N_E \Z)^*$ such that a prime $d\nmid N_E$ is contained in $\pi_{\Omega}$ if and only if the reduction of $-d$ modulo $N_E$ is contained in $\mathfrak{r}_{\Omega}$.
Furthermore, the cardinality of $\mathfrak{r}_{\Omega}$ is given by
\[
\# r_{\Omega}=\frac{\varphi(N_E)}{2^{\omega(N_E)}}.
\]
\end{lemma}
\begin{proof}
The prime $d\nmid N_E$ is contained in $\pi_{\Omega}$ if and only if $-d$ is not a quadratic residue modulo $q$ for all primes $q\in \Omega$ and it is a quadratic residue modulo $q$ for all primes $q\in \Pi(N_E)\setminus \Omega$.
The number of quadratic residues in $(\Z/q\Z)^*$ is equal to the number of quadratic non-residues.
For $q\in \Omega$ (resp. $q\notin \Omega$) denote by $\mathfrak{r}_{\Omega}^{q}\subset (\Z/q\Z)^*$ the set of quadratic non-residues (resp. residues).
Define
\[
\mathfrak{r}_{\Omega}:=\prod_{q\in \Pi(N_E)} \mathfrak{r}_{\Omega}^q\subset \prod_{q\in \Pi(N_E)} (\Z/q\Z)^*=(\Z/N_E \Z)^*.
\]
Note that $d\in \pi_{\Omega}$ if and only if the mod-$N_E$ reduction $-\bar{d}\in \mathfrak{r}_{\Omega}$.
Since $\#\mathfrak{r}_{\Omega}^q=\frac{q-1}{2}$, we have that
\[
\#\mathfrak{r}_{\Omega}=\prod_{i=1}^t \frac{r_i-1}{2}=\frac{\varphi(N_E)}{2^{\omega(N_E)}}.\]
\end{proof}

\noindent We are now in a position to prove Theorem $\ref{th62}$.
\begin{proof}[Proof of Theorem $\ref{th62}$]
Recall that in order for (CR) to be satisfied, it is required that $N^-_{E,d}$ is a product of an odd number of primes and none of the primes $r_i$ for $i\leq k$ divide $N^-_{E,d}$.
Let $\Omega$ be any subset of $\{r_{k+1}, \dots, r_t\}$.
The set $\pi_{\Omega}$ consists of the primes $d$ such that $N^-_{E,d}=\prod_{q\in \Omega} q$.
Let $\mathcal{N}$ be the collection of subsets $\Omega\subseteq \{r_{k+1}, \dots, r_t\} $ for which $\# \Omega$ is odd, and set $\mathcal{S}':=\sqcup_{\Omega\in \mathcal{N}} \pi_{\Omega}$.
The sets $\mathcal{S}$ and $\mathcal{S}'$ differ by a finite set of primes.
Lemma $\ref{residueclasses}$ asserts that $\pi_{\Omega}$ is determined by $\frac{\varphi(N_E)}{2^{\omega(N_E)}}$ residue classes modulo $N_E$.
Therefore, by Dirichlet's theorem, each residue class contributes a density of $\frac{1}{\varphi(N_E)}$.
We find that
\[
\delta_{\Omega}=\frac{1}{\varphi(N_E)}\frac{\varphi(N_E)}{2^{\omega(N_E)}}=\frac{1}{2^{\omega(N_E)}}.
\]
Therefore, $\delta(\mathcal{S})=\frac{\#\mathcal{N}}{2^{\omega(N_E)}}$.

It remains to calculate $\#\mathcal{N}$, i.e., the number of odd subsets of $\{r_{k+1}, \dots, r_{\omega(N_E)}\}$.
It is an exercise in combinatorics that $\# \mathcal{N}=2^{\omega(N_E)-k-1}$, see \cite[Exercise 1.1.13]{soberon2013}.
The claim follows.
\end{proof}

\begin{remark}
If $N_E$ is even and square-free, then we need to be careful when working with $q=2$.
The ramification behavior of 2 in $\Q(\sqrt{-d})$ is not detected by a congruence condition on $d \pmod{2}$ and instead we need to consider the residue class of $d \pmod{8}$.
\end{remark}

In fact, the set $\delta(\mathcal{S})$ can be computed explicitly in any specific example using mostly elementary methods.
This is illustrated in the example below.
\begin{example}
Let $E_{/\Q}$ be the curve with Cremona label \href{https://www.lmfdb.org/EllipticCurve/Q/11a1/}{\tt 11a1}, and let us take $p=7$.
In the notation of the theorem, we have
\begin{itemize}
\item $N_E=11$ is square-free
\item $\bar{\rho}_{E,p}$ is surjective
\item $a_p(E)=-2 \not\equiv \pm 1 \pmod p$
\item $q =11 \not\equiv \pm 1 \pmod p$
\end{itemize}
Thus, the hypotheses of Theorem \ref{th62} are satisfied with $k=0$, and so we conclude that $\delta(\mathcal{S})=\frac{1}{2}$.
In particular, since $N=11$ is prime, we have
\begin{align*}
\mathcal{S} &= \{ \text{primes}\ d \ \vert \ (E,p,K^d) \ \text{satisfies (CR)} \}\\
&=\{ \text{primes}\ d \ \vert 11 \ \text{is inert in}\ K^d / \Q \} \\
&=\left\{ \text{primes}\ d \ \vert \left( \frac{11}{d}\right)=-1 \right\}.
\end{align*}
Some calculations with quadratic reciprocity and the Chinese Remainder Theorem show that
\[
\left( \frac{11}{d}\right)=-1 \Leftrightarrow d \equiv 3, 13, 15, 17, 21, 23, 27, 29, 31, \text{ or } 41 \pmod{44}.
\]
This description gives 10 possible congruence classes for $d$.
We claim that this agrees with our calculation $\delta(\mathcal{S})=\frac{1}{2}$.
First note that any odd prime falls into an odd congruence class modulo $44$, cutting our space to $22$ eligible congruence classes.
But a prime $d$ cannot be congruent to $11$ or $33$ either, so the number of eligible congruence classes for $d$ is just $20=\varphi(44)$, justifying our claim.
\end{example}

\section{Results for varying imaginary quadratic number field \texorpdfstring{$K$}{}: the rank zero case}
\label{section: vary K rank 0}

Let $E_{/\Q}$ be a rank 0 elliptic curve of conductor $N_E$ without complex multiplication.
We fix an odd prime number $p$ at which $E$ has good ordinary reduction.
Let $d\nmid N_E$ be a prime number and $K^d = \Q(\sqrt{-d})$ be an imaginary quadratic field.
We assume throughout that the Tate--Shafarevich group $\Sh(E/K^d)$ is finite.

Fix a pair $(E,p)$ and let $d$ vary over the primes coprime to $N_E$.
If the Mordell--Weil rank of $E(K^d)$ is zero\footnote{In view of Goldfeld's Conjecture, we expect that this happens 50\% of the time.}, then $\Sel_{p^{\infty}}(E/K^d)$ is finite and the Euler characteristic $\chi_d:=\chi(K_{\infty}^d/K^d, E[p^{\infty}])$ is defined.
It satisfies the relation
\[
\chi_d\sim \frac{\left(\alpha_p^{(d)}\right)^2\times \left(\Sh_p^{(d)}\times \prod_{v\in S_{\ns}^d\setminus S_p^d}\tau_v^{(d)} \right)}{\left(\# E(K^d)[p^\infty]\right)^2},
\]
where $\Sh_p^{(d)}=\Sh_p(E/K^d)$, $\alpha_v^{(d)}=\alpha_v(E/K^d)$, and $\tau_v^{(d)}=c_v^{(p)}(E/K^d)$.
Also, $S_{\ns}^d$ is the set of primes of $K^d$ that are finitely decomposed in $K_{\infty}^d$ and $S_p^d$ consists of the primes of $K^d$ above $p$.
Since $\Sel_{p^\infty}(E/K^d)$ is finite in our setting, it follows from \cite[Corollary 4.9]{Gre98_PCMS} that the Selmer group $\Selpell$ is cotorsion as a $\Lambda$-module.
By Proposition \ref{prop36}, we know that $\chi_d=1$ if and only if $\mu_p^{(d)}:=\mu_p(E/K_{\infty}^d)=0$ and $\lambda_p^{(d)}:=\lambda_p(E/K_{\infty}^d)=0$.
In this case, $\Selpell$ has finite cardinality and since it does not contain any proper finite index submodules, it must equal $0$.
As $d$ varies over all prime numbers we study the following question

\begin{question}What is the density of primes $d\nmid N_E$ such that $\rank_{\Z} E(K^d)=0$ and the following equivalent conditions are satisfied.
\begin{enumerate}
 \item $I^{(d)}=1$ for each of the invariants
 \[I^{(d)}\in\{\tau_v^{(d)},\alpha_p^{(d)}, \Sh_p^{(d)}\},\]
 \item $\chi_d=1$,
 \item $\mu_p^{(d)}=0$, and $\lambda_p^{(d)}=0$,
 \item $\Selpell=0$?
\end{enumerate}
\end{question}

In this section, we analyze the variation of the invariants $I^{(d)}$ above as $d$ ranges over all primes.
Denote by $\Xi_{E,p}(x)$, (resp. $\Sh_{E,p}(x)$) the number of primes $d\leq x$ such that the Tamagawa product $\tau_{E,p}^{(d)}:=\prod_{v\in S_{\ns}^d\setminus S_p^d} \tau_v^{(d)}$ (resp. $\Sh_p^{(d)}$) is \emph{not} equal to $1$.
We seek to understand the growth of the functions $\Xi_{E,p}(x)$ and $\Sh_{E,p}(x)$ as $x\rightarrow \infty$.
Unlike the invariants $\tau_v^{(d)}$ and $\Sh_p^{(d)}$, the quantity $\alpha_p$ is a lot more well-behaved.
While there are many known results about Tamagawa products, there is little known unconditionally regarding the function $\Sh_{E,p}$.
However, we are able to provide numerical data with regards to the growth of this function.

Let us introduce the setting in which it is most convenient to apply this method of analysis.
Let $\Q_{\infty}$ be the \emph{cyclotomic} $\Z_p$-extension of $\Q$ and $\Sel_{p^{\infty}}(E/\Q_{\infty})$ the $p$-primary Selmer group over $\Q_{\infty}$.
Since $E(\Q)$ is assumed to be finite, under the additional assumption that $\Sh(E/\Q)$ is finite, the Selmer group $\Sel_{p^{\infty}}(E/\Q)$ is finite and the associated Euler characteristic is defined.
Let $\Gamma_{\Q}:=\Gal(\Q_{\infty}/\Q)$ and set
\[
\chi(\Gamma_{\Q}, E[p^{\infty}]):=\chi\left(\Gamma_{\Q}, \Sel_{p^{\infty}}(E/\Q_{\infty})\right).
\]
The Euler characteristic formula is given by (see \cite[chapter 3]{CS00_GCEC}
\begin{equation}
\label{eqn: ECF cyclotomic over Q}
\chi(\Gamma_{\Q}, E[p^{\infty}])\sim \frac{\alpha_p^2\times \Sh_p \times \prod_{\ell}\tau_{\ell}}{\left(\# E(\Q)[p^\infty]\right)^2}.
\end{equation}
Here, $\alpha_p, \Sh_p$ and $\tau_\ell$ are as follows:\begin{itemize}
 \item $\alpha_p:=\#\widetilde{E}(\F_p)[p]$,
 \item $\Sh_p:=\# \Sh(E/\Q)[p]$,
 \item $\tau_\ell:=c_\ell^{(p)}$.
\end{itemize}
The following conditions are known to be equivalent (see \cite[Corollary 3.6]{KR21})
\begin{enumerate}
 \item $\alpha_p=1, \Sh_p=1$, and $\tau_{\ell}=1$ for all primes $\ell$,
 \item $\chi(\Gamma_{\Q}, E[p^{\infty}])=1$,
 \item $\mu(E/\Q_{\infty})=0$ and $\lambda(E/\Q_{\infty})=0$,
 \item $\Sel_{p^{\infty}}(E/\Q_{\infty})=0$.
\end{enumerate}
Arguments in \emph{loc. cit.} show that it is reasonable to expect that the above conditions are satisfied for most\footnote{We mean that for a fixed prime $p$, the proportion of rank zero elliptic curves satisfying one (and hence all) of the above conditions depends on $p$ and as $p\to \infty$, this proportion approaches $1$.} elliptic curves of rank zero (when the prime $p$ is reasonably large).
We assume that these conditions are indeed satisfied.
\subsection{}
Throughout, $p\geq 5$ is fixed and suppressed in the notation.
For a rational prime $\ell\neq p$, set $\tau_{\ell}^{(d)}:=\prod_{v| \ell} \tau_v^{(d)}$, where the product is taken for the primes $v| \ell$.
In this section, we study the variation of the Tamagawa numbers $\tau_{\ell}^{(d)}$ as $d$ varies over all primes.
Recall the assumption that $\tau_{\ell}=1$, i.e., $p$ does not divide the Tamagawa number $c_{\ell}$.
By \cite[p. 448]{silverman2009}, $c_\ell$ is divisible by $p\geq 5$ precisely when the Kodaira type of $E$ at $\ell$ is $\textrm{I}_n$ with $p|n$.
Therefore, our assumption is equivalent to assuming that the Kodaira type of $E$ at any prime $\ell$ is \emph{not} of the type $\textrm{I}_n$ for $n$ divisible by $p$.

\begin{theorem}\label{th72}
Let $E_{/\Q}$ be an elliptic curve with good ordinary reduction at a prime number $p\geq 5$ and conductor $N_E$.
Let $\ell$ be a prime such that $\ell| N_E$ and $\tau_{\ell}=1$.
Then the following assertions hold.
\begin{enumerate}
\item If $\ell\neq 2$ then $\tau_{\ell}^{(d)}=1$ for $d\neq \ell$.
\item If $\ell=2$ then $\tau_{\ell}^{(d)}= 1$ for all primes $d$.
\end{enumerate}
Therefore, $\Xi_{E,p}(x)\leq \omega(N_E)$ for all values of $x>0$.
\end{theorem}

\begin{proof}
Fix the prime $\ell$, denote by $T_{\base}$ the Kodaira symbol of $E$ over $\Q_\ell$.
In \cite{kida03}, M.~Kida studied the variation of the reduction type under a finite extension $K_v/\Q_{\ell}$.
For $p\geq 5$, the Kodaira symbol of the base change $E_{/K_v}$ is determined by $T_{\base}$ and the ramification index of $K_v/\Q_{\ell}$, as we explain.
\begin{enumerate}
\item First consider the case when $\ell \neq 2$.
Let $K:= K^d$ as $d$ varies over all primes.
As we vary over all extensions $K_v$ with $v|\ell$, we read off the ``new" Kodaira type $T_{\new}$ from \cite[Table 1, p. 556-7]{kida03}.
Recall that the ramification index of $K_v/ \Q_{\ell}$ is either 1 (when unramified) or 2 (when totally ramified).
Note that $K/\Q$ is a quadratic extension, hence if $\ell$ ramifies in $K$, it must be tamely ramified.
According to \cite[Table 1, p. 556-7]{kida03}, upon base change from $\Q_\ell$ to $K_v$, the Tamagawa number $c_{v}$ can become divisible by by $p$ in the following two cases:
\begin{enumerate}
\item If $T_{\base} = \textrm{I}_n$ such that $p|n$.
\item If $T_{\base} = \textrm{I}_n^*$ such that $p|n$ and $\ell$ is ramified in $K_v$.
\end{enumerate}
Our starting assumption was that $\tau_{\ell} = 1$.
Equivalently, $T_{\base}\neq \textrm{I}_n$ for $p|n$.
Therefore, given an elliptic curve $E$ with bad reduction at $\ell \neq 2, \ p$ and Kodaira type $\textrm{I}_n^*$, we need to count the number of $d$'s such that $\ell$ ramifies in $K^d$.
Since $d$ varies only over primes, there is exactly one such imaginary quadratic field, i.e., $K = \Q(\sqrt{-\ell})$.
\item Next, suppose that $\ell=2$.
If $\ell$ is split in $K=K^d$, then $K_v = \Q_{\ell}$, and $\tau_\ell^{(d)}=\tau_\ell=1$.
Hence, assume that $\ell$ is either inert or ramified in $K$.
When $2$ is inert (hence unramified), \cite[Table 1, p. 556-7]{kida03} asserts that $p| c_v$ (for $v|p$) precisely when $T_{\base} = \textrm{I}_n$ for some $p|n$.
Our hypothesis rules out this case.
When $2$ ramifies (wildly), we refer to \cite[Table 3, p. 559]{kida03}.
In this case, $p|c_{v}$ if and only if $T_{\base}=\textrm{I}_n$ when $p|n$.
However, our starting assumption guarantees that this does not happen.
Therefore, given an elliptic curve $E$ with bad reduction at $\ell =2$ and Kodaira type not equal to $\textrm{I}_n$ with $p|n$, we see that $\tau_{\ell}^{(d)}=1$ for all primes $d$.
\end{enumerate}
\end{proof}

\subsection{}
We fix a rank 0 elliptic curve $E_{/\Q}$ and a prime $p\geq 5$.
Assume that the Selmer group $\Sel_{p^{\infty}}(E/\Q_{\infty})=0$.
Varying over all primes $d$, we study how often does $\Sel_{p^{\infty}}(E/K_{\infty}^d)=0$.
Throughout, we impose the following assumption.
\begin{hypothesis}
For each imaginary quadratic field $K^d$, the $p$-primary part of the Tate--Shafarevich group $\Sh(E/K^d)$ is finite.
\end{hypothesis}
We are only able to prove a partial result in this effect, which shows that the question is largely dependent on how often $\Sh_p^{(d)}\neq 1$ as $d$ varies over all primes.
\begin{theorem}
\label{thm: partial result for varying K}
Let $E_{/\Q}$ be an elliptic curve with conductor $N_E$ and assume the following.
\begin{enumerate}[(i)]
\item $p\geq 5$.
\item $\rank_{\Z} E(\Q)=0$.
\item $E(\Q)[p]= 0$.
\item Write $a_p(E):=1+p-\#\widetilde{E}(\F_p)$ and suppose that $a_p(E)\not \equiv \pm 1 \pmod{p}$.
\item The following equivalent conditions are satisfied:
\begin{enumerate}
\item the $\mu$ and $\lambda$-invariants of $\Sel_{p^{\infty}}(E/\Q_{\infty})$ are zero,
\item $\Sel_{p^{\infty}}(E/\Q_{\infty})=0$.
\end{enumerate}
\end{enumerate}
Let $d\nmid N_E$ be a prime such that $\rank_{\Z} E(K^d)=0$.
Then the following conditions are equivalent
\begin{enumerate}
\item $\Sel_{p^{\infty}}(E/K_{\infty}^d)=0$,
\item $\Sh(E/K^d)[p]=0$.
\end{enumerate}
\end{theorem}

\begin{proof}
Since $\rank_{\Z} E(\Q)=0$, it follows that $\Sh(E/\Q)$ is finite and $\Sel_{p^{\infty}}(E/\Q_{\infty})$ has well-defined Euler characteristic $\chi(\Gamma_{\Q}, E[p^{\infty}])$.
Since $\Sel_{p^{\infty}}(E/\Q_{\infty})=0$, it follows that $\chi(\Gamma_{\Q}, E[p^{\infty}])=1$.
Therefore, according to \eqref{eqn: ECF cyclotomic over Q} $\Sh_p:=\#\Sh(E/\Q)[p^{\infty}]=1$ and $\tau_\ell:=c_\ell^{(p)}=1$ for all primes $\ell\neq p$.

Recall that $\Sel_{p^{\infty}}(E/K_{\infty}^d)$ contains no proper finite-index $\Lambda$-submodules; it is finite if and only if it is $0$.
Therefore, Proposition \ref{prop36} asserts that $\Sel_{p^{\infty}}(E/K_{\infty}^d)=0$ if and only if its Euler characteristic $\chi_d:=\chi(\Gamma_{K^d}, E[p^{\infty}])$ is equal to $1$.
Since $a_p\not \equiv \pm 1\pmod{p}$, it follows that $p\nmid \#\widetilde{E}(\F_{p^2})$.
According to Theorem \ref{th72}, $\tau_\ell^{(d)}=1$ for all primes $\ell$.
It follows from \eqref{ecf}, that the following relation holds
\[
\chi_d=\#\Sh(E/K^d)[p^{\infty}].
\]
The result is now immediate.
\end{proof}

Conjecturally, one can analytically determine whether $\Sh(E/K^d)[p]=0$ for any specific example.
More precisely \cite{lmfdb}, the Birch and Swinnerton-Dyer Conjecture gives a formula for $\# \Sh(E/K^d)$ as in \eqref{eq:bsd},
the computation of which has been implemented in Magma \cite{magma}.

In Table \ref{tab:sha}, we present the results of this computation for a handful of elliptic curves defined over $\Q$ with small, square-free conductor.
The first column of the table identifies an elliptic curve by its Cremona label.
For each curve, we determined which primes $5 \leq d \leq 150$ satisfied both $(N_E,d)=1$ and $\rank_{\Z} E(K^d)=0$; this information is given in the second column of the table.
For each such $d$, we used Magma to compute $\# \Sh(E/K^d)$ as given by \eqref{eq:bsd}.
The third column gives the number of $d$ such that $\Sh(E/K^d)[p] \neq0$ for $p=2$, and so on.

While this data is not enough to allow for any conjectures, we note that (assuming BSD), the information in Table \ref{tab:sha} and the conclusion of Theorem \ref{thm: partial result for varying K} suggest, for instance, that $\Sel_{p^{\infty}}(E/K_{\infty}^d)=0$
for $E=${\href{https://www.lmfdb.org/EllipticCurve/Q/11a1/}{\tt 11a1}} and $p \geq 5$ for each of the values of $d$ we tested.

\section{Results for varying elliptic curves}
\label{section: vary elliptic curve}
In this section, we fix a prime $p\geq 5$ and an imaginary quadratic field $K$.
Recall that $K_{\infty}$ is the anticyclotomic $\Z_p$-extension of $K$ and $\Q_{\infty}$ the cyclotomic $\Z_p$-extension of $\Q$.
Set $\Gamma_\Q$ (resp. $\Gamma_K$) to denote $\Gal(\Q_{\infty}/\Q)$ (resp. $\Gal(K_{\infty}/K)$).
We study the variation of Iwasawa invariants of the Selmer groups $\Sel_{p^{\infty}}(E/\Q_{\infty})$ and $\Sel_{p^{\infty}}(E/K_{\infty})$ as $E$ ranges over all elliptic curves defined over $\Q$, with $\rank_{\mathbb{Z}}E(K)=0$, with good ordinary reduction at $p$.
Since the Mordell-Weil rank is zero, it follows that the Selmer group $\Sel_{p^\infty}(E/K_\infty)$ is a cotorsion $\Lambda$-module.
In this section, we do not impose the hypothesis (CR).
The results in this section extend those in \cite[\S4]{KR21}.

Recall that any elliptic curve $E_{/\Q}$ admits a unique Weierstrass equation
\begin{equation}\label{weier}
E:Y^2 = X^3 + AX + B
\end{equation}
where $A, B$ are integers and $\gcd(A^3 , B^2)$ is not divisible by any twelfth power.
Since $p\geq 5$, such an equation is minimal.
We order elliptic curves by height and expect that similar results shall hold when they are ordered by conductor or discriminant.
Recall that the \emph{height of} $E$ satisfying the minimal equation $\eqref{weier}$ is given by $H(E) := \max\left(\abs{A}^3, B^2\right)$.

For $x>0$, let $\mathcal{E}$ be the set of isomorphism classes of elliptic curves defined over $\Q$, and $\mathcal{E}(x)$ the number of elliptic curves $E$ defined over $\Q$ for which $H(E)\leq x$.
If $\mathcal{S}$ is a subset of $\mathcal{E}$, set $\mathcal{S}(x)=\mathcal{S}\cap \mathcal{E}(x)$.
Recall the Euler characteristic formula for the cyclotomic $\Z_p$-extension from \eqref{eqn: ECF cyclotomic over Q},
\[
\chi(\Gamma_{\Q}, E[p^{\infty}])\sim \frac{\# \Sh(E/\Q)[p^{\infty}]\times \left(\# \widetilde{E}(\F_p)\right)^2 }{\left(\# E(\Q)[p^\infty]\right)^2}\cdot \prod_{\ell} c_{\ell}^{(p)}(E/\Q).
\]
On the other hand, recall that the anticyclotomic Euler characteristic formula $\eqref{ecf}$ states that
\[
\chi(\Gamma_K, E[p^{\infty}])\sim \frac{\# \Sh(E/K)[p^{\infty}]\times \left(\prod_{\p|p}\# \widetilde{E}(\kappa_\p)\right)^2}{\left(\# E(K)[p^\infty]\right)^2}\cdot \prod_{v\in S_{\ns}\setminus S_p} c_v^{(p)}(E/K).
\]
The key observation in this section is that to analyze the variation of the Euler characteristic (and hence $\mu$ and $\lambda$-invariants) of elliptic curves, it suffices to study the average behaviour of the following quantities for fixed $p$ and varying $E\in \mathcal{E}$.
We consider the following terms
\begin{itemize}
 \item $\Sh_p(E/\Q):=\#\Sh(E/\Q)[p^{\infty}]$ and $\Sh_p(E/K) := \#\Sh(E/K)[p^{\infty}]$,
 \item $\tau_p(E/\Q):=\prod_{\ell} c_\ell^{(p)}(E/\Q)$ and $\tau_p(E/K) :=\prod_{v\in S_{\ns}\setminus S_p}c_v^{(p)}(E/K)$,
 \item $\alpha_p(E/\Q):=\#\widetilde{E}(\F_p)[p^{\infty}]$, $\alpha_p(E/K) :=\prod_{\p|p}\#\widetilde{E}(k_\p)[p^{\infty}]$.
\end{itemize}
\begin{definition}
Let $E_{/\Q}$ be an elliptic curve, we say that $E$ satisfies $(\dagger)$ if the Kodaira type at $\ell=2,3$ is not of of the form $I_n$ for some $n$ divisible by $p$.
\end{definition}
Note that an elliptic curve with good reduction at $\ell=2,3$ satisfies $(\dagger)$.
\begin{definition}
Let $\mathcal{E}_{1,F}(x)$, $\mathcal{E}_{2,F}(x)$, and $\mathcal{E}_{3,F}(x)$ be the subset of elliptic curves $E\in \mathcal{E}(x)$ satisfying $(\dagger)$, for which $p$ divides $\Sh_p(E/F)$, $\tau_p(E/F)$, and $\alpha_p(E/F)$ respectively for $F\in \{\Q,K\}$.
\end{definition}
Note that no assumptions are made on the rank of elliptic curves in $\mathcal{E}(x)$ or $\mathcal{E}_{i,F}(x)$.

\begin{definition}For $F\in\{\Q, K\}$, let $\mathcal{J}_F\subset \mathcal{E}$ consist of the elliptic curves $E$ satisfying the following conditions:
\begin{enumerate}
 \item $E$ satisfies $(\dagger)$,
 \item $\rank_{\Z} E(F)=0$,
 \item $E$ has good ordinary reduction at the primes above $p$.
\end{enumerate}
\end{definition}

\subsection{} To motivate our results for the anticyclotomic $\Z_p$-extension $K_{\infty}/K$, we briefly recall similar analysis that has been over the cyclotomic $\Z_p$-extension of $\Q$ in \cite[\S4]{KR21}.
There are a few differences in notation, primarily because we only consider the $p$-ordinary case here.
However, the results in the $p$-supersingular case (over $\Q_{\infty}$) are analogous and are contained in \emph{loc. cit.}

\begin{definition}
Let $\mathcal{Z}_{\Q}\subset \mathcal{J}_{\Q}$ be the subset of elliptic curves $E$ for which the following equivalent conditions are satisfied.
\begin{enumerate}
 \item The Selmer group $\Sel_{p^{\infty}}(E/\Q_{\infty})$ has $\mu=0$ and $\lambda=0$,
 \item The Selmer group $\Sel_{p^{\infty}}(E/\Q_{\infty})=0$,
 \item The Euler characteristic $\chi(\Gamma_\Q, E[p^{\infty}])=1$.
\end{enumerate}
Denote by $\mathcal{Z}^c_{\Q}$ its complement in $\mathcal{J}_{\Q}$.
\end{definition}
In the following result, we give an upper bound for $\limsup_{x\rightarrow \infty} \frac{\#\mathcal{Z}^c_{\Q}(x)}{\#\mathcal{E}(x)}$ depending only on $p$.
It is a direct consequence of the Euler characteristic formula \eqref{eqn: ECF cyclotomic over Q}.

\begin{lemma}
With respect to notation above, we have that
\[\limsup_{x\rightarrow \infty} \frac{\#\mathcal{Z}^c_{\Q}(x)}{\#\mathcal{E}(x)}\leq \sum_{i=1}^3 \limsup_{x\rightarrow \infty}\frac{\#\mathcal{E}_{i,\Q}(x)}{\#\mathcal{E}(x)}.\]
\end{lemma}

\begin{proof}
Since the Euler characteristic formula must be an integer, it follows from \eqref{eqn: ECF cyclotomic over Q} that if $\Sh_p(E/\Q)$, $\tau_p(E/\Q)$ and $\alpha_p(E/\Q)$ are all $1$, then the Euler characteristic $\chi(\Gamma_{\Q}, E[p^{\infty}])=1$.
By \cite[Corollary 3.6]{KR21}, we know that
\[\chi(\Gamma_{\Q}, E[p^{\infty}])=1\Leftrightarrow \Sel_{p^{\infty}}(E/\Q_{\infty})=0,\]and the result follows.
\end{proof}

There are conjectural bounds due to Delaunay which predict that the quantity $\limsup_{x\rightarrow \infty}\frac{\#\mathcal{E}_{1,\Q}(x)}{\#\mathcal{E}(x)}$ becomes really small as $p$ becomes large.
We refer the reader to \cite{Del01} for further details.
In \cite{KR21}, explicit bounds are obtained for $\limsup_{x\rightarrow \infty}\frac{\#\mathcal{E}_{i,\Q}(x)}{\#\mathcal{E}(x)}$ for $i=2,3$, which indicate that these quantities approach $0$ quickly as $p$ gets large.

\begin{definition}\label{def85}
For any prime $\ell$, write $\cE_{\ell}^T(x)$ for the subset of $\cE(x)$ consisting of elliptic curves with bad reduction at $\ell$ and Kodaira type $T$.
\end{definition}

The following result characterizes the density of $\cE_{\ell}^{\textrm{I}_n}(x)$.
\begin{theorem}
\label{thm: In case of sadek}
For any prime $\ell\neq 2,3,p$, and any integer $n\geq 1$,
\[
\limsup_{x\rightarrow \infty} \frac{\#\cE_{\ell}^{\textrm{I}_n}(x)}{\#\cE(x)}\leq \frac{(\ell-1)^2}{\ell^{n+2}}.
\]
\end{theorem}

\begin{proof}
The result follows verbatim from the argument in \cite[Theorem 4.11]{KR21}, which proves the statement for $I_p$.
It has come to our attention that the additional assumption $\ell\neq 2,3$ is required. 
\end{proof}

As a corollary, we obtain the following result.

\begin{corollary}
With notation as above,
\begin{equation}\label{87eq}
\limsup_{x\rightarrow \infty} \frac{\#\cE_{2,\Q}(x)}{\#\cE(x)}=\sum_{\ell\neq 2,3,p} \frac{(\ell-1)^2}{\ell^2(\ell^p -1)}
<\sum_{\ell\neq 2,3,p} \frac{1}{\ell^p}
<\zeta(p)-1,\end{equation} where the sum is taken over primes $\ell\neq 2,3,p$.
\end{corollary}

\begin{proof}
By \cite[p. 448]{silverman2009}, the Tamagawa number $c_\ell(E)$ is divisible by $p$ precisely when the Kodaira type of $E_{/\Q_\ell}$ is of the form $\textrm{I}_m$ for some integer $m\geq 1$ which is divisible by $p$.
Note that the primes $\ell=2,3$ are excluded since $E\in \mathcal{E}_{2,\Q}(x)$ satisfies $(\dagger)$ by assumption.
Therefore, one has the bound
\[
\begin{split}\limsup_{x\rightarrow \infty} \frac{\#\cE_{2,\Q}(x)}{\#\cE(x)}
&\leq \sum_{\ell\neq 2,3,p}\sum_{n=1}^{\infty} \limsup_{x\rightarrow \infty} \frac{\#\cE_{\ell}^{\textrm{I}_{np}}(x)}{\#\cE(x)}\\
&\leq \sum_{\ell\neq 2,3,p}\sum_{n=1}^{\infty} \frac{(\ell-1)^2}{\ell^{pn+2}}\\
& = \sum_{\ell\neq 2,3,p} \frac{(\ell-1)^2}{\ell^2}\cdot\frac{1}{\ell^p-1}.
\end{split}
\]
Moreover,
\[\begin{split}\frac{(\ell-1)^2}{\ell^2}\cdot \frac{1}{\ell^p-1}&=\left(1-\frac{2}{\ell}+\frac{1}{\ell^2}\right)\frac{1}{\ell^p-1}\\
&<\left(1-\frac{1}{\ell^p}\right)\frac{1}{\ell^p-1}=\frac{1}{\ell^p}.
\end{split}\]
The result follows.
\end{proof}

\begin{remark}
We clarify that there is a minor inaccuracy in the formula in \cite{KR21}, since there the sum is taken only for the Kodaira type $\textrm{I}_p$, when in fact it needs to be taken over all $\textrm{I}_n$ where $p| n$.
The above estimate \eqref{87eq} is the correct one.
\end{remark}

For $x>0$, denote by $\mathcal{W}(x)$ the set of Weierstrass equations for which the height is $\leq x$.
For $\kappa=(a,b)\in \F_p\times \F_p$ with $\Delta(\kappa):=4a^3+27b^2$ non-zero, associate the elliptic curve $E_{\kappa}$ defined by the Weierstrass equation
\[
E_{\kappa}:Y^2=X^3+aX+b.
\]
Note that $\kappa$ is not uniquely determined by the isomorphism class of $E_{\kappa}$.
Denote by $\mathfrak{S}_p$ the set of pairs $\kappa=(a,b)\in \F_p\times \F_p$ such that $E_{\kappa}$ contains a point of order $p$ over $\F_p$.
Set $d(p):=\# \mathfrak{S}_p$.
Let $\mathcal{W}'(x)\subset \mathcal{W}(x)$ be the set of Weierstrass equations $Y^2=X^3+AX+B$ which reduce to $E_{\kappa}$ for some $\kappa\in \mathfrak{S}_p$.

\begin{theorem}
With notation as above,
\[
\limsup_{x\rightarrow \infty} \frac{\#\cE_{3,\Q}(x)}{\#\cE(x)}\leq \zeta(10)\cdot \frac{d(p)}{p^2}.
\]
\end{theorem}

\begin{proof}
See \cite[Theorem 4.10]{KR21}.
\end{proof}

Putting it all together, we obtain the following result.
\begin{theorem}
With notation as above,
\[
\limsup_{x\rightarrow \infty} \frac{\#\mathcal{Z}^c_{\Q}(x)}{\#\mathcal{E}(x)}< \limsup_{x\rightarrow \infty} \frac{\#\cE_{1,\Q}(x)}{\#\cE(x)}+ \sum_{\ell\neq 2,3,p} \frac{1}{\ell^p}+\zeta(10)\cdot \frac{d(p)}{p^2}.
\]
\end{theorem}
One expects that the term $\frac{\cE_{1,\Q}(x)}{\cE(x)}$ decreases rapidly as $p\rightarrow \infty$, as is explained in the next section.

\subsection{}
Let $K$ be a fixed imaginary quadratic field and $p\geq 5$ a fixed prime.
Recall that $K_{\infty}$ is the anticyclotomic $\Z_p$-extension of $K$.
If $E_{/\Q}$ is an elliptic curve for which $\rank_{\Z} E(K)=0$ and $\Sh(E/K)$ is finite, then the Selmer group $\Sel_{p^{\infty}}(E/K)$ is finite, and $\Selp$ is cofinitely generated and cotorsion as a $\Lambda$-module (see \cite[Corollary 4.9]{Gre98_PCMS}).
We assume throughout that $\Sh(E/K)$ is finite for the elliptic curves considered.

\begin{lemma}
\label{boringlemma}
Let $E_{/\Q}$ be an elliptic curve such that $\rank_{\Z} E(K)=0$ and $\Sh(E/K)$ is finite, then the following conditions are equivalent.
\begin{enumerate}
 \item The Euler characteristic $\chi(\Gamma_K, E[p^{\infty}])=1$.
 \item The Selmer group $\Sel_{p^{\infty}}(E/K_{\infty})$ has $\mu=0$ and $\lambda=0$.
 \item The Selmer group $\Sel_{p^{\infty}}(E/K_{\infty})$ has finite cardinality.
 \item The Selmer group $\Sel_{p^{\infty}}(E/K_{\infty})=0$.
\end{enumerate}
\end{lemma}

\begin{proof}
The equivalence of conditions (1)-(3) follows from Proposition \ref{prop36}.
Since the Selmer group $\Sel_{p^{\infty}}(E/K_{\infty})$ does not contain any proper finite index submodules, it follows that it must equal $0$.
Hence, (3) is equivalent to (4).
\end{proof}

\begin{definition}
Let $\mathcal{Z}_{K}\subset \mathcal{J}_{K}$ be the subset of elliptic curves $E_{/\Q}$ for which the equivalent conditions of Lemma \ref{boringlemma} hold.
Denote by $\mathcal{Z}^c_{K}$ its complement in $\mathcal{J}_{K}$.
\end{definition}
It follows from the Euler characteristic formula \eqref{ecf} that \begin{equation}\limsup_{x\rightarrow \infty} \frac{\#\mathcal{Z}^c_{K}(x)}{\#\mathcal{E}(x)}\leq \sum_{i=1}^3 \limsup_{x\rightarrow \infty}\frac{\#\mathcal{E}_{i,K}(x)}{\#\mathcal{E}(x)}.\end{equation}
We prove explicit upper bounds for the quantities $\limsup_{x\rightarrow \infty}\frac{\#\mathcal{E}_{i,K}(x)}{\#\mathcal{E}(x)}$ for $i=2,3$.
We show these limits approach $0$ as $p\rightarrow \infty$.
After establishing these bounds, we present Cohen--Lenstra heuristics for the variation of Tate--Shafarevich groups which give us a good idea of what to expect for $\limsup_{x\rightarrow \infty}\frac{\#\mathcal{E}_{1,K}(x)}{\#\mathcal{E}(x)}$.
These heuristics also apply to Tate--Shafarevich groups over $\Q$, thereby further improving on the arguments in \cite[ \S4]{KR21}.

We record the following well-known fact (see for example \cite[pp. 2132-2133]{Bri07}).
\begin{fact}
Let $\ell\neq p$ be a prime.
Then the primes of $K$ above $\ell$ are non-split in $K_{\infty}$ if and only if $\ell$ splits in $K$.
\end{fact}

If $v|\ell$ and $v\in S_{\ns}\backslash S_p$, then the Kodaira type of $E_{/\Q_\ell}$ is the same as that of $E_{/K_v}$ as $K_v=\Q_\ell$.
Thus, $p|c_v(E)$ if and only if $E_{/\Q_\ell}$ has Kodaira type $\textrm{I}_{pn}$.
Let $\mathfrak{s}_K$ be the set of primes $\ell\neq 2,3,p$ which split in $K=\Q(\sqrt{-d})$, i.e., $\ell\neq 2,3,p$ and $\left(\frac{-d}{p}\right)=1$.
We make the following observation.

\begin{lemma}\label{lemma814}
 With respect to prior notation,
\[\limsup_{x\rightarrow\infty} \frac{\#\cE_{2,K}(x)}{\#\cE(x)}\leq \sum_{\ell\in \mathfrak{s}_K} \sum_{n=1}^{\infty} \limsup_{x\rightarrow\infty} \frac{\#\cE_\ell^{\textrm{I}_{np}}(x)}{\#\cE(x)},\]
where $\cE_\ell^{\textrm{I}_{np}}(x)$ is as in Definition \ref{def85}.
\end{lemma}

\begin{theorem}
Let $p\geq 5$ be a fixed prime.
Let $K$ be a fixed imaginary quadratic field and $\mathfrak{s}_K$ the set of primes $\ell\neq p$ that split in $K$.
Then
\begin{align*}
\limsup_{x\rightarrow \infty}\frac{\#\cE_{2,K}(x)}{\#\cE(x)}
&\leq \sum_{\ell\in\mathfrak{s}_K} \frac{(\ell-1)^2}{\ell^2(\ell^p -1)} < \sum_{\ell\in \mathfrak{s}_K} \frac{1}{\ell^p}<\zeta(p)-1.
\end{align*}
\end{theorem}

\begin{proof}
Combining Lemma \ref{lemma814} and Theorem \ref{thm: In case of sadek}, we obtain
\begin{align*}
\limsup_{x\rightarrow \infty} \frac{\#\cE_{2,K}(x) }{\#\cE(x)}
&= \limsup_{x\rightarrow\infty}\sum_{\ell\in \mathfrak{s}_K} \sum_{n\geq 1} \frac{\#\cE_{\ell}^{\textrm{I}_{np}}(x)}{\#\cE(x)}\\
&\leq \sum_{\ell\in \mathfrak{s}_K} \frac{(\ell-1)^2}{\ell^2(\ell^p -1)}\\
& < \sum_{\ell\in \mathfrak{s}_K} \frac{1}{\ell^p}\\
&< \zeta(p) -1.
\end{align*}
\end{proof}

Finally, we analyze the term $\limsup_{x\rightarrow \infty}\frac{\#\cE_{3,K}(x)}{\#\cE(x)}$.
There is a dichotomy that arises depending on whether $p$ is inert or non-inert in $K$.
If $p$ is non-inert in $K$, the analysis is the same as in that in \cite{KR21}.
This is because in this case, the residue field $k_v$ of $K$ at any prime $v|p$ is equal to $\F_p$.
However, if $p$ is inert, there are some minor modifications since in this case, the residue field is $\F_{p^2}$.

\begin{lemma}
Let $E_{/\Q}$ be an elliptic curve and $a_p(E):=1+p-\#\widetilde{E}(\F_p)$.
Then $p$ divides $\#\widetilde{E}(\F_{p^2})$ if and only if $a_p(E)\equiv \pm 1\pmod{p}$, or equivalently, $\#\widetilde{E}(\F_p)\equiv 0,2\pmod{p}$.
\end{lemma}

\begin{proof}
Arguing as in the proof of Lemma \ref{lemma46}, we find that
\[
\# \widetilde{E}(\F_{p^2})=(p+1-a_p)(p+1+a_p),
\]
from which the result is immediate.
\end{proof}

Denote by $\mathfrak{T}_p$ the set of pairs $\kappa=(a,b)\in \F_p\times \F_p$ such that $\# E_{\kappa}(\F_{p})\equiv 0, 2\pmod{p}$.
Set
\[
b(p):=\begin{cases}\# \mathfrak{T}_p\text{ if } p \text{ is inert in }K,\\
\# \mathfrak{S}_p\text{ otherwise.}
\end{cases}
\]
The proof of the following result is identical to that of \cite[Theorem 4.10]{KR21}.
\begin{theorem}
With notation as above,
\[\limsup_{x\rightarrow \infty} \frac{\#\cE_{3,K}(x)}{\#\cE(x)}\leq \zeta(10)\cdot \frac{b(p)}{p^2}.
\]
\end{theorem}
Putting it all together, we obtain the following result.
\begin{theorem}
\label{main result varying elliptic curve}
With notation as before,
\[\limsup_{x\rightarrow \infty} \frac{\#\mathcal{Z}^c_{K}(x)}{\#\mathcal{E}(x)}< \limsup_{x\rightarrow \infty} \frac{\#\cE_{1,K}(x)}{\#\cE(x)}+ \sum_{\ell\in \mathfrak{s}_K} \frac{1}{\ell^p}+\zeta(10)\cdot \frac{b(p)}{p^2}.\]
\end{theorem}
In Section \ref{section: CL heuristics for Sha}, we provide heuristics for the limits $\limsup_{x\rightarrow\infty}\frac{\#\cE_{1,\Q}(x)}{\#\cE(x)}$ and $\limsup_{x\rightarrow\infty}\frac{\#\cE_{1,K}(x)}{\#\cE(x)}$.
The approximate values of $\#\mathfrak{S}_p/p^2$ for the primes $7\leq p< 150$, up to 16 decimal places are noted in Table \ref{tab:2}.
One observes that they seem to be going to $0$ as $p$ gets large, though there is much oscillation in the data.
Furthermore, explicit calculation indicates that $\#\mathfrak{T}_p=2\#\mathfrak{S}_p$.
It was pointed out to us by L.~C.~Washington that this is the case since the curves with $a_p=1$ are quadratic twists of curves for which $a_p=-1$, hence there are the same number of each.
This explains why one has the equality $\#\mathfrak{T}_p=2\#\mathfrak{S}_p$.

\section{Cohen--Lenstra Heuristics for Tate--Shafarevich Groups}
\label{section: CL heuristics for Sha}

In \cite{CL84}, Cohen and Lenstra formulated heuristics for the variation of class groups of number fields.
Delaunay extended their approach to formulate heuristics for Tate--Shafarevich groups for elliptic curves $E_{/\Q}$ ordered by height, see \cite{Del01}.
These heuristics have been refined in \cite{bhargava13modeling} and this is the subject of much interest in arithmetic statistics of elliptic curves.
In this section, we survey such heuristics and explain their consequences to our results in the previous section.

Let $\mathscr{E}$ denote the set of isomorphism classes of elliptic curves defined over $\Q$ with rank $0$.
For $x>0$, set $\mathscr{E}(x)$ to be the subset of $\mathscr{E}$ consisting of $E$ such that $H(E)\leq x$.
Note that there are the following inclusions
\[\mathcal{J}_K\subset \mathcal{J}_{\Q}\subset \mathscr{E}\subset \mathcal{E}.\]
Delaunay's heuristics apply to the set of elliptic curves $\mathscr{E}$.
It is reasonable to expect that the same heuristics should apply to $\mathcal{J}_K$ or $\mathcal{J}_{\Q}$, since the guiding principle is the same and independent of the extra conditions defining $\mathcal{J}_F$.

First we discuss the variation of $\Sh(E/\Q)$ as $E$ varies over $\mathscr{E}$.
For an elliptic curve $E_{/\Q}$, J.~W.~S.~Cassels showed that there is a bilinear alternating pairing
\[
\Sh(E/\Q)\times \Sh(E/\Q)\rightarrow \Q/\Z,
\]
which is non-degenerate if $\Sh(E/\Q)$ is finite.
If $E\in \mathscr{E}$, i.e., $\rank_{\Z}E(\Q)=0$, then it is known that $\Sh(E/\Q)$ is finite.
This motivates the notion of a group of type-$S$.
\begin{definition}
Let $G$ be a finite abelian group.
We say that $G$ is of \emph{type-$S$} if there is a non-degenerate alternating pairing
\[\beta:G\times G\rightarrow \Q/\Z.\]
\end{definition}
If $G$ is a group of type-$S$, then the pairing $\beta$ is unique.
Two groups of type-$S$ are isomorphic if there is an isomorphism which preserves the pairing.
There is a neat description of such groups.
If $G$ is a group of type-$S$, it is easy to show that $G\simeq H\times H$ for a finite abelian group.
Conversely, if $H$ is a finite abelian group and $G\simeq H\times H$, then there is a unique non-degenerate alternating pairing $\beta:G\times G\rightarrow \Q/\Z$.
Given a finite abelian group $G$, set $G_p$ to denote the $p$-torsion subgroup.
The quantity $r_p(G) = \dim_{\F_p} (G_p)$ is known as the \emph{$p$-rank}.
Set $\Aut_S(G)$ to denote the group of automorphisms of $G$ which preserve the pairing $\beta$.

Let $\mathscr{G}$ denote the set of isomorphism classes of groups of type-$S$ and consider a function $F:\mathscr{G}\rightarrow \mathbb{C}$.
There is a natural map
\[
\Sh: \mathscr{E}\rightarrow \mathscr{G}
\]
which sends an elliptic curve $E$ to its Tate--Shafarevich group.
The following proportion is a measure of the function $F\circ \Sh$ on $\mathscr{E}$ on average
\[
M_{\mathscr{E}}(F):=\lim_{x\rightarrow \infty} \frac{\sum_{E\in \mathscr{E}(x)} F(\Sh(E))}{\#\mathscr{E}(x)}.
\]
A priori it is not clear that this limit must exist.
However, it is conjectured to exist in the cases of interest in this paper.
We remark that the average in \cite{Del01} is taken with respect to the conductor and not height
It is expected that the same heuristics apply, regardless (see \cite[footnote on p. 3]{bhargava13modeling}).

The key idea of Cohen and Lenstra is that class groups behave as random finite abelian groups $G$ except that they have to be weighted by $\abs{\Aut(G)}^{-1}$.
The same should be true for Tate--Shafarevich groups of elliptic curves, where the average is taken over all groups of type-$S$.
For $n\in \Z_{\geq 1}$, let $\mathscr{G}(n)$ be the set of isomorphism classes of groups $G\in \mathscr{G}$ with order equal to $n$.
Following Delaunay \cite[Definition 11]{delaunay2007heuristics}, for $\alpha\in \mathbb{R}_{\geq 1}$, define the following density associated to $F$, taken over $\mathscr{G}$
\[
M_{\mathscr{G}}(F,\alpha):=\lim_{x\rightarrow \infty} \frac{\left(\sum_{n\leq x}\sum_{\mathscr{G}(n)}\frac{F(G)\abs{G}^{\alpha}}{\abs{\Aut_S(G)}}\right)}{\left(\sum_{n\leq x}\sum_{\mathscr{G}(n)}\frac{\abs{G}^{\alpha}}{\abs{\Aut_S(G)}}\right)}.
\]
This quantity is expected to be independent of $\alpha$ (as stated in \textit{loc. cit.}) and a choice $\alpha=1$ is made.
Write $M_{\mathscr{G}}(F):=M_{\mathscr{G}}(F,1)$.
\begin{heuristic}[Delaunay \cite{delaunay2007heuristics}]
\label{delaunayheuristic}
For all reasonable functions $F:\mathscr{G}\rightarrow \C$, the limits $M_{\mathscr{E}}(F)$ and $M_{\mathscr{G}}(F)$ converge and
\[M_{\mathscr{E}}(F)=M_{\mathscr{G}}(F).\]
\end{heuristic}

There is an analogous conjecture for elliptic curves $E_{/\Q}$ of rank $1$ and the heuristics in this setting are different.
Since our results apply only to elliptic curves of rank 0, we do not further explain such heuristics here.
We note that the distinction between rank 0 and 1 elliptic curves comes across clearly from the heuristics in \cite{bhargava13modeling}.

Now we return our applications from the previous section.
Consider the function
\[
F_{p-\op{triv}}(G):=\begin{cases} 1\text{ if }G_p=0,\\
0\text{ otherwise.}
\end{cases}
\]
On computing $M_{\mathscr{G}}(F_{p-\op{triv}})$, we have the following expectation.

\begin{theorem}[Delaunay \cite{Del01}]
Assume that the Heuristic \ref{delaunayheuristic} is satisfied for the function $F=F_{p-\op{triv}}$.
Then
\[
\lim_{x\rightarrow \infty} \frac{\#\{E\in \mathscr{E}(x)\mid \Sh(E/\Q)[p]\neq 0\}}{\#\mathscr{E}(x)}=f_0(p),
\]
where $f_0(p)$ is given by
\[f_0(p)=1-\prod_{j=1}^{\infty} \left(1-\frac{1}{p^{2j-1}}\right)=\frac{1}{p}+\frac{1}{p^3}-\frac{1}{p^4}+\frac{1}{p^5}-\frac{1}{p^6}\dots.\]
\end{theorem}
In particular, $f_0(2)\approx 0.58$, $f_0(3)\approx 0.36$ and $f_0(5)\approx 0.21$.
These heuristics cannot be checked via numerical computation since there are too many elliptic curves and Tate--Shafarevich groups which appear for large heights.
It is reasonable to expect that the same heuristics should apply to the averages $\limsup_{x\rightarrow\infty}\frac{\#\cE_{1,\Q}(x)}{\#\cE(x)}$ and $\limsup_{x\rightarrow\infty}\frac{\#\cE_{1,K}(x)}{\#\cE(x)}$.
First note that
\[
\limsup_{x\rightarrow\infty}\frac{\# \cE_{1,\Q}(x)}{\# \cE(x)}=\limsup_{x\rightarrow \infty} \frac{\#\{E\in \mathcal{J}_{\Q}(x) \mid \Sh(E/\Q)[p]\neq 0\}}{\#\mathscr{E}(x)}.
\]
The extra $p$-ordinary and $(\dagger)$ conditions defining $\mathcal{J}_{\Q}$ as a subset of $\mathscr{E}$ should be independent of the heuristics for the variation of $\Sh(E/\Q)$.
However, since $\mathcal{J}_{\Q}$ is a subset of $\mathscr{E}$, the above heuristics imply that
\[
\limsup_{x\rightarrow\infty}\frac{\# \cE_{1,\Q}(x)}{\# \cE(x)}\leq f_0(p).
\]
On the other hand,
\[
\limsup_{x\rightarrow\infty}\frac{\# \cE_{1,K}(x)}{\# \cE(x)}=\limsup_{x\rightarrow \infty} \frac{\#\{E\in \mathcal{J}_{K}(x) \mid \Sh(E/K)[p]\neq 0\}}{\#\mathscr{E}(x)},
\]
and it is reasonable to expect that the groups $\Sh(E/K)$ should also obey the Cohen--Lenstra heuristics over $\mathcal{J}_K$.
Therefore, we expect that both limits $\lim_{x\rightarrow\infty}\frac{\#\cE_{1,\Q}(x)}{\#\cE(x)}$ and $\lim_{x\rightarrow\infty}\frac{\#\cE_{1,K}(x)}{\#\cE(x)}$ should exist and be $\leq f_0(p)$.
However, we are reluctant to make this into a conjecture since these expectations cannot be verified numerically.

\section{Indefinite Case}
\label{sec:indefinite} Recall that we decomposed the conductor of $E$ as $N=N^+ N^-$, where $N^+$ (resp. $N^-$) is divisible only by primes which are split (resp. inert) in $K$.
When $N^-$ is divisible by an odd number of primes, we are said to be working in the \textit{definite case}, and when $N^-$ is divisible by an even number of primes, it is called the \textit{indefinite case}.
There is a sharp divide in anticyclotomic Iwasawa theory between these two cases due to the presence of Heegner points in the latter setting.
For the rest of the paper, we focus on the indefinite case.
More specifically, we assume $N^-=1$. This assumption is necessary for some of our arguments (e.g. when applying the results of \cite{HatleyLeiJNT}), but we expect similar results to hold in the more general indefinite setting.

The most significant difference, is that in the indefinite case, the Selmer group $\Sel_{p^\infty}(E/K_\infty)$ is \emph{not always} $\Lambda$-cotorsion (see \cite[Theorem A]{Ber95}).
Thus, many of the arguments we gave in the preceding sections fail, and the corresponding results are often false.
For instance, in Lemma \ref{lemma21} we proved that $\lambda_p(E/K_\infty) \geq \mathrm{rank}_\Z E(K)$ by considering the short exact sequence
\[
0 \rightarrow E(K) \otimes \Qp/\Zp \rightarrow\Sel_{p^\infty}(E/K) \rightarrow \Sh(E/K)[p^\infty] \rightarrow 0.
\]
If $\Sel_{p^\infty}(E/K_\infty)$ is cotorsion, then $\mathrm{corank}_{\Zp} \Sel_{p^\infty}(E/K)$ is finite and the exact sequence implies
\[
\mathrm{corank}_{\Zp} \Sel_{p^\infty}(E/K) \geq \mathrm{rank}_\Z E(K).
\]
However, this argument fails when $\Sel_{p^\infty}(E/K_\infty)$ has positive corank.
In this setting, one is able to characterize the $\lambda$-invariant of the Selmer group $\Sel_{p^\infty}(E/K_\infty)$.
Indeed, we show below in Theorem \ref{thm:indefinite-main} that, under certain hypotheses, we can have $\mathrm{rank}_\Z E(K) =1$ and $\lambda_p(E/K_\infty)=0$.

As mentioned in the discussion following Theorem \ref{thm:van-ord}, there is no known formula for the Euler characteristic of $\Sel_{p^\infty}(E/K_\infty)$ in the indefinite setting.
We circumvent this issue by using recent progress towards the anticyclotomic Iwasawa Main Conjectures made by A.~Burungale--F.~Castella--C.~H.~Kim \cite{BCK} to obtain an Euler characteristic formula for an auxiliary Selmer group (see Section \ref{sec:Selmer-definitions} for the definition).
The $\lambda$-invariant of this auxiliary Selmer group provides an upper bound for the $\lambda$-invariant of $\Sel_{p^\infty}(E/K_\infty)$ (see Lemma \ref{lem:lambda-one-trickl}).
This allows us to proceed with our usual methods for computing Iwasawa invariants on average when $\mathrm{rank}_\Z E(K)=1$.
This argument is quite novel and we expect that this sort of argument could eventually be used to shed light on Iwasawa invariants in various other contexts where the Selmer groups are not $\Lambda$-cotorsion.

\subsection{Indefinite Case: Assumptions}
\label{sec:notation}
For the rest of the paper, we will alter our assumptions on $E$ and $K$.
We consider elliptic curves $E_{/\Q}$ of Mordell--Weil rank $r=1$ with good ordinary reduction at $p$ of square-free conductor $N>3$.
Also, assume throughout that the $p$-primary Tate--Shafarevich group $\Sh(E/K)[p^\infty]$ is finite.
Let $K/\Q$ be an imaginary quadratic extension with group of units $\{ \pm 1\}$.
Let $h_K$ denote the class number of $K$.
We fix embeddings $K \hookrightarrow \mathbb{C}$ and $\overline{\Q} \hookrightarrow \overline{\Q}_p$.
Let $\bar{\rho}_{E,p} \colon G_\Q \rightarrow \GL_2(\F_p)$ denote the residual Galois representation determined by $E[p]$.

\begin{definition}\label{def:admiss}
The triple $(E,K,p)$ is called \emph{admissible} if the following hold:
\begin{enumerate}
\item $p=\p \bar{\p}$ splits in $K$,
\item every prime divisor of $N_E$ splits in $K$,
\item $E$ has good ordinary reduction at $p$,
\item $p \nmid 6N\varphi(N)h_K$,
\item $a_p(E) \not\equiv \pm 1 \pmod p$,
\item $\bar{\rho}_{E,p}$ is surjective,
\item $\bar{\rho}_{E,p}$ is ramified at $\ell$ for each $\ell \mid N_E$.
\end{enumerate}
\end{definition}
\noindent Here, $\varphi$ denotes the Euler totient function and $h_K$ is the class number of $K$.
Later in this section, we fix an elliptic curve $E_{/\Q}$ without complex multiplication and an imaginary quadratic field $K$ in which $p$ splits and all prime divisors $\ell|N_E$ split in $K$.
It is then shown that for $1/2$ the primes $p$, the above conditions are satisfied.

\subsection{Auxiliary Selmer Groups}
\label{sec:Selmer-definitions}
We will now define some auxiliary Selmer groups which are useful in the indefinite case.
These Selmer groups are defined by prescribing the local conditions at the primes above $p$ either more or less strict than those defining the usual Selmer group $\Sel_{p^\infty}(E/K_\infty)$.
We retain the notation from Section \ref{section: background}.

Let $L$ be a number field such that $K \subset L \subset K_\infty$.
For $v \in S_p$ and $\mathcal L_v \in \{ \emptyset, \Gr, 0 \}$, define
\[
\mathcal{J}_{\mathcal{L}_v}(E/L) := \begin{cases}
0 & \text{if}\ \mathcal L_v = \emptyset, \\
J_v(E/L) & \text{if}\ \mathcal L_v = \Gr, \\
H^1\left( L_v, E\right) & \text{if}\ \mathcal L_v = 0.
\end{cases}
\]
Note that when $\mathcal{L}_v=\emptyset$, there is no condition at $v$, hence it is called the empty condition.
The condition $\mathcal{L}_v=\op{Gr}$ is called the Greenberg condition and $\mathcal{L}_v=0$ is the strict condition.
Then for any pair $\mathcal{L}=\{ \mathcal{L}_{\p}, \mathcal{L}_{\bar{\p}} \}$, we define the an associated $p$-primary Selmer group.
For $v|\p$ (resp. $v|\bar{\p}$), let $\mathcal{L}_v=\mathcal{L}_\p$ (resp. $\mathcal{L}_v=\mathcal{L}_{\bar{\p}}$).
Define the $p$-primary $\mathcal{L}$\textit{-Selmer group over} $L$ as follows
\[
\Sel_{{\mathcal{L}},p^\infty}(E/L):=\ker\left\{ H^1\left(K_S/L,E[p^{\infty}]\right)\longrightarrow \left(\bigoplus_{v\in S \setminus S_p} J_v(E/L)\right)\oplus\left( \bigoplus_{v\in S_p} \mathcal{J}_{\mathcal{L}_v}(E/L)\right) \right\} .
\]
Just as before, we then define
\[J_{\mathcal{L}_v}(E/\Kinf):=\varinjlim_{L\subset \Kinf} J_{\mathcal{L}_v}(E/L),\]
and the \emph{$p$-primary $\mathcal{L}$-Selmer group over $\Kinf$} is defined as follows
\[
\Sel_{{\mathcal{L}},p^\infty}(E/K_\infty):=\ker\left\{ H^1\left(K_S/K_\infty,E[p^{\infty}]\right)\longrightarrow \left(\bigoplus_{v\in S \setminus S_p} J_v(E/K_\infty)\right)\oplus \left( \bigoplus_{v\in S_p} \mathcal{J}_{\mathcal{L}_v}(E/K_\infty)\right)\right\}.
\]

\begin{remark}
Some possible choices for $\mathcal{L}$ lead to familiar objects of study:
\begin{enumerate}[(a)]
\item When $\mathcal{L}=\{\Gr,\Gr\}$, this definition recovers the usual $p$-primary Selmer group, and in this case, we omit $\mathcal{L}$ from the notation.
\item When $\mathcal{L}=\{0,0\}$, this definition recovers the \textit{fine} Selmer group (see \cite{CS05}).
\item The Selmer group obtained from $\mathcal{L}=\{\emptyset,0\}$ is, by the Iwasawa Main Conjecture, an algebraic analogue of the $p$-adic $L$-function constructed by M.~Bertolini--H.~Darmon-K.~Prasanna \cite{BDP}.
See, for instance, \cite[Conjecture 1.2]{KobOta} or \cite[Conjecture 2.1]{AC}.
\end{enumerate}
\end{remark}


For each $\mathcal{L}$, write $\X_\cL(E/K_\infty)$ for the Pontryagin dual of $\Sel_{{\mathcal{L}},p^\infty}(E/K_\infty)$.

\begin{proposition}
\label{prop:coranks}
With notation as above,
\[
\rank_\Lambda\X_\cL(E/K_\infty)=
\begin{cases}
0&\text{if }\cL=\{\emptyset,0\},\{\Gr,0\},\\
1&\text{if }\cL=\{\Gr,\Gr\},\{\Gr,\emptyset\}.
\end{cases}
\]
\end{proposition}

\begin{proof}
As noted in \cite[Corollary 3.8]{HL-MRL} this follows from \cite[Corollary 3.3.4]{howard} and results in \cite{Castella}.
\end{proof}

For each $\cL$, write $\mu_{p,\cL}(E)$ for the corresponding $\mu$-invariant, and similarly for the $\lambda$-invariants.
In many cases, these $\mu$-invariants vanish.

\begin{proposition}
\label{prop:mu-vanishes}
Suppose that the triple $(E,K,p)$ is admissible in the sense of Definition \ref{def:admiss}.
If
\[
\mathcal{L} \in \{\{\emptyset,0\}, \{\Gr,0\}, \{\emptyset,\Gr\}, \{\Gr,\Gr\} \},
\]
then $\mu_{p,\cL}(E)=0$.
\end{proposition}

\begin{proof}
See \cite[Theorem 4.5]{HatleyLeiJNT}.
\end{proof}

\begin{lemma}
\label{lem:lambda-one-trickl}
Suppose the the triple $(E,K,p)$ is admissible in the sense of Definition \ref{def:admiss}.
Then $\lambda_p(E) \leq \lambda_{p,\{ \emptyset,0\}}(E)$.
\end{lemma}

\begin{proof}
By dualizing the natural inclusion $\Sel_{\{\Gr,0\},p^\infty}(E/K_\infty) \hookrightarrow \Sel_{\{\emptyset,0\},p^\infty}(E/K_\infty)$, we obtain a surjective map
\[
\mathcal{X}_{\{\emptyset,0\}}(E) \twoheadrightarrow \mathcal{X}_{\{\Gr,0\}}(E).
\]
Since each of these groups is $\Lambda$-torsion by Proposition \ref{prop:coranks}, the kernel of this map must also be $\Lambda$-torsion.
Since $\lambda$-invariants are non-negative and are additive in short exact sequences where the first term is a torsion module (see for example, \cite[Proposition 2.1]{HL-MRL}), we deduce that $\lambda_{p,\{ \Gr,0\}}(E)\leq \lambda_{p,\{ \emptyset,0\}}(E)$.
But it has been shown in \cite[Lemma 2.3]{Castella} that $\lambda_{p,\{ \emptyset,\Gr\}}(E) = \lambda_{p,\{ \Gr,0\}}(E)$.
Therefore,
\[
\lambda_{p,\{ \emptyset,\Gr\}}(E) = \lambda_{p,\{ \Gr,0\}}(E)\leq \lambda_{p,\{ \emptyset,0\}}(E).
\]
Upon dualizing the natural inclusion $\Sel_{p^\infty}(E/K_\infty) \hookrightarrow \Sel_{\{\emptyset,\Gr\},p^\infty}(E/K_\infty)$, we obtain
\begin{equation}\label{eq:gr-empty-to-gr-gr}
0 \rightarrow C \rightarrow \mathcal{X}_{\{\emptyset,\Gr\}}(f) \rightarrow \mathcal{X}(f) \rightarrow 0,
\end{equation}
where $C$ is a torsion $\Lambda$-module.
Applying the aforementioned result on additivity of the $\lambda$-invariant (\cite[Proposition 2.1]{HL-MRL}) yields $\lambda_p(E)\leq \lambda_{p,\{ \emptyset,\Gr\}}(E)$.
This completes the proof of the lemma.
\end{proof}

\subsection{A criterion for the triviality of \texorpdfstring{$\lambda_p(E/K_\infty)$}{} }
Recall from Proposition \ref{prop:mu-vanishes} that our running hypotheses imply that $\mu_{p}(E/K_\infty)=0$.
We will now establish a criterion for showing that $\lambda_{p}(E/K_\infty)=0$ as well. 
This criterion relies on the deep connection between $\Sel_{\{\emptyset,0\},p^\infty}(E/K_\infty)$ and the Bertolini--Darmon--Prasanna $p$-adic $L$-function of \cite{BDP}. Recall that in this setting, the main conjecture is proven in \cite{BCK}, and as a result, one has a formula for the constant term of the characteristic element of the Selmer group $\Sel_{\{\emptyset,0\},p^\infty}(E/K_\infty)$.

\begin{proposition}
\label{prop:indefinite-main-conjecture}
Assume that $\Sh(E/K)[p^\infty]$ is finite and that $(E,K,p)$ is admissible.
Let $f_{\{\emptyset,0\}}(T)$ denote a generator of the characteristic ideal of the torsion module $\mathcal{X}_{\{\emptyset,0\}}(E/K_\infty)$, viewed as an element of $\Z_p\llbracket T\rrbracket$.
Then
\[
f_{\{\emptyset,0\}}(0) \sim \left(\frac{1-a_p+p}{p}\cdot \log_{\omega_E}(P)\right)^2.
\]
Here $\log_{\omega_E}$ is the $p$-adic logarithm associated to the N\'{e}ron differential $\omega_E$, and $P \in E(K)$ is a point of infinite order given by the Gross--Zagier formula.
\end{proposition}

\begin{proof}
Since $(E,K,p)$ is admissible, the hypotheses of \cite[Proposition A.3 and Theorem A.4]{BCK} are satisfied; hence the result follows.
We note that in \textit{loc. cit.} $E$ is assumed to be an optimal curve in its isogeny class, but since $N_E$ is square-free, replacing $E$ with any isogenous curve scales the value of $\log_{\omega_E}(P)$ by the corresponding Manin constant, and this is a $p$-adic unit by \cite[Corollary 4.1]{Mazur-RationalIsog}.
\end{proof}

\begin{theorem}\label{thm:indefinite-main}
Let $E_{/\Q}$ be an elliptic curve, $K$ an imaginary quadratic field whose unit group is $\{\pm 1\}$ and $p$ an odd prime.
Furthermore, assume that the following conditions are satisfied:
\begin{enumerate}
 \item $\rank_{\mathbb{Z}}E(K)=1$ and $\Sh(E/K)[p^\infty]$ is finite,
 \item $(E,K,p)$ is admissible in the sense of Definition \ref{def:admiss}.
\end{enumerate}Let $P \in E(K)$ be a point of infinite order.
Then $\mu_p(E/K_\infty)=0$ and $\lambda_p(E/K_\infty)=0$ if and only if $\log_{\omega_E}(P)/p$
is a $p$-adic unit.
\end{theorem}

\begin{proof}
Since $(E,K,p)$ is admissible, we know that $a_p(E) \not\equiv \pm 1 \pmod p$, and so by Proposition \ref{prop:indefinite-main-conjecture}, $f_{\{\emptyset,0\}}(0)$ is a $p$-adic unit if and only if the same is true of $\log_{\omega_E}(P)/p$.
But since $f_{\{\emptyset,0\}}(T)$ is a distinguished polynomial, this is equivalent to $f_{\{\emptyset,0\}}(T)=1$, which is in turn equivalent to the simultaneous vanishing of $\lambda_{p,\{\emptyset,0\}}(E)$ and $\mu_{p,\{\emptyset,0\}}(E)$.
The result now follows from Lemma \ref{lem:lambda-one-trickl}.
\end{proof}

The question of when $\left(\log_{\omega_E}(P)/p\right)$ is a $p$-adic unit seems to be a subtle one.
When $E[p]$ is reducible, this quantity was shown to be a unit under many circumstances in \cite{KrizLi}, but in the irreducible case much less appears to be known.
However, this quantity is sometimes amenable to direct computation (as in \cite[$\S$6]{KrizLi}).
In Table \ref{table:heegner}, we record some data in this direction, which suggests that this term is indeed a $p$-adic unit more often than not.
As pointed out to us by H.~Darmon, one might expect that $\left(\log_{\omega_E}(P)/p\right)$ is \emph{not} a unit $1/p$ times.
The Sage code used to do these computations is available upon request.

\subsection{Statistics}
\label{stats for indefinite}
In the remainder of this paper, we apply Theorem \ref{thm:indefinite-main} to study questions in arithmetic statistics.
In particular, we study how often the assumptions of the theorem are satisfied, and thus how often the Iwasawa invariants $\mu_p(E/K_\infty)$ and $\lambda_p(E/K_\infty)$ are $0$ in the indefinite setting.
Recall that the triple $(E,K,p)$ is called \emph{admissible} if the following are satisfied:
\begin{enumerate}
\item $p=\p \bar{\p}$ splits in $K$
\item Every prime divisor of $N_E$ splits in $K$
\item $E$ has good ordinary reduction at $p$
\item $p \nmid 6N\varphi(N)h_K$
\item $a_p(E) \not\equiv \pm 1 \pmod p$
\item $\bar{\rho}_{E,p}$ is surjective
\item $\bar{\rho}_{E,p}$ is ramified at $\ell$ for each $\ell \mid N_E$.
\end{enumerate}

\subsubsection{Fix $E$ and $K$, vary $p$.}
In this section, we fix an elliptic curve $E_{/\Q}$ and an imaginary quadratic field $K=\Q(\sqrt{-d})$ and vary the prime $p$.
Recall that it is assumed that the unit group of $K$ is $\{\pm 1\}$.
We make the following additional assumptions on the pair $(E,K)$.
\begin{enumerate}
 \item $E$ is not CM and is semi-stable,
 \item $\rank_{\mathbb{Z}} E(K)=1$,
 \item every prime $\ell$ dividing the conductor $N_E$ of $E$ splits in $K$.
\end{enumerate}

\begin{theorem}
Let $(E,K)$ be as above, then for $1/2$ of the primes $p$ at which $E$ has good 
reduction, the following are equivalent:
\begin{enumerate}
 \item $\mu_p(E/K_\infty)=0\text{ and } \lambda_p(E/K_\infty)=0$,
 \item $\log_{\omega_E}(P)/p$ is a $p$-adic unit, where $P$ is a generator of $E(K)$ prescribed by the Gross--Zagier formula.
\end{enumerate}
\end{theorem}

\begin{proof}
Since $\rank_{\mathbb{Z}}E(K)=1$, Theorem \ref{thm:indefinite-main} asserts that the above statements are equivalent when $(E,K,p)$ is admissible.
We show that the triple $(E,K,p)$ is admissible for $1/2$ of the primes $p$.
\begin{enumerate}
 \item It is a consequence of the Chebotarev density theorem that $1/2$ of primes $p$ split in $K$.
 \item We assume that every prime divisor of $N_E$ splits in $K$.
 This statement is independent of $p$.
 \item By the work of J.~P.~Serre (see \cite{Ser81}) we know that $E$ has good ordinary reduction at $p$ for $100\%$ of the primes.
 This can be made more precise.
 Note that $E$ has good ordinary reduction at $p$ precisely when $a_p:=p+1-\#\widetilde{E}(\F_p)$ is not divisible by $p$.
 When $p\geq 7$, this is the case precisely when $a_p\neq 0$.
 Given any number $a\in \Z$, the set of primes $p$ such that $a_p=a$ makes up $0\%$ of the primes.
 In fact, letting $\pi_a(x)$ be the number of primes $p\leq x$ such that $a_p=a$, it is known unconditionally that $\pi_a(x)\leq \frac{(\log\log x)^2}{(\log x)^2}$ as $x\rightarrow \infty$ (see \cite{Mur97}).
 \item Clearly, $p \nmid 6N\varphi(N)h_K$ is satisfied for all but finitely many primes $p$.
 \item It follows from the Hasse bound that for $p$ large enough, $a_p\equiv \pm 1\pmod{p}$ if and only if $a_p=\pm 1$.
 For any integer $a$, since $\pi_a(x)$ is $o(x)$, it follows that $a_p\not\equiv \pm 1\pmod{p}$ for $100\%$ of the primes $p$.
 \item It follows from Serre's open image theorem (see \cite[\S4 Theorem 3]{Serre72}) that $\bar{\rho}_{E,p}$ is surjective for all but finitely many primes $p$.
 In fact, Mazur has shown that for a semi-stable elliptic curve without CM, $\bar{\rho}_{E,p}$ is surjective for $p\geq 13$.
 \item Let $\ell|N_E$, since $E$ has semi-stable reduction at $\ell$, the set of primes $p\neq \ell$ for which $\bar{\rho}_{E,p}$ is unramified at $\ell$ is finite.
 Once $p$ is big enough so that the $j$-invariant $j(E)$ is a $p$-adic unit, by the theory of the Tate curve, $\bar{\rho}_{E,p}$ will be ramified at $\ell$.
 Hence, $\bar{\rho}_{E,p}$ will be ramified at all primes $\ell|N_E$ when $p$ is big enough.
 This argument was pointed out to us by R.~Ramakrishna.
\end{enumerate}
Putting it all together, it follows that $(E,K,p)$ is admissible for $1/2$ of primes $p$, and hence Theorem \ref{thm:indefinite-main} applies to $1/2$ of the primes $p$.
\end{proof}

\begin{remark}
The density of integral
Weierstrass equations which are minimal models of semi-stable elliptic curves over $\Q$ (that is, elliptic curves with square-free conductor) is $1/\zeta(2) \approx 60.79\%$ and the density of semi-stable elliptic curves over $\Q$ is $\zeta(10)/\zeta(2) \approx 60.85\%$, see \cite{cremona2020local}.
\end{remark}

\subsubsection{Fix $E$ and $p$, vary $K$}
We fix a prime $p\geq 5$ and a \emph{semi-stable} elliptic curve $E_{/\Q}$.
For simplicity, assume that $E$ has prime conductor $q(\neq p)$ and satisfies the following properties:
\begin{enumerate}
\item $E$ has good ordinary reduction at $p$,
\item $a_p(E)\not\equiv \pm 1\pmod{p}$,
\item $\varphi(q)= q-1\not \equiv 0 \pmod{p}$,
\item $\bar{\rho}_{E,p}$ is surjective,
\item $\bar{\rho}_{E,p}$ is ramified at $\ell$ for each $\ell \mid N_E$.
\end{enumerate}
Let $d$ be a positive square-free integer such that the unit group of $K^d=\Q(\sqrt{-d})$ is $\{\pm 1\}$.
To understand how often the pair $(E,K,p)$ is admissible, we are left to study the following two questions
\begin{enumerate}
 \item\label{condition 1} For what proportion of $K=\Q(\sqrt{-d})$ does $p\nmid h_K$.
 \item\label{condition 2} For what proportion of $K$ do both the primes $p$ and $q$ split in $K$.
\end{enumerate}
Recalling that the discriminant of a quadratic field uniquely identifies it.
More precisely, the number of imaginary quadratic fields up to discriminant $x$ is $\lfloor x\rfloor $, i.e., the greatest integer $\leq x$.

The first question is related to the Cohen--Lenstra heuristics.
More precisely, for a square-free positive integer $d$, define
\[
N_{p}(x) = \# \{d: p|h_{K} \textrm{ and } \abs{D_{K}}\leq x \}.
\]
Then the Cohen--Lenstra heuristics predict that (see \cite[Section 5]{BM20})
\[
N_p(x)\sim \frac{6}{\pi^2}\left(1 - \prod_{j\geq 1} \left( 1 - \frac{1}{p^j}\right)\right)x.
\]
We now answer the second question.

\begin{lemma}
Given distinct primes $p$ and $q$, the number of \emph{imaginary quadratic fields} (with $\abs{D_K}<x$) in which both $p$ and $q$ split completely is given by
\[
\frac{pq}{8(1+p)(1+q)\zeta(2)}x + O(pq x^{1/2}).
\]
\end{lemma}

\begin{proof}
See \cite[Lemma 2.2 and Proposition 8.1]{EPW17}.
\end{proof}

It would seem as though the conditions \eqref{condition 1} and \eqref{condition 2} are independent of each other, however, we are not able to prove such a result.
If the conditions are independent, then the proportion of imaginary quadratic fields for which both conditions are satisfied is
\[\frac{6}{\pi^2}\left(1 - \prod_{j\geq 1} \left( 1 - \frac{1}{p^j}\right)\right)\cdot \left(\frac{pq}{8(1+p)(1+q)\zeta(2)}\right).\]
We are thus not able to give a definite answer to the question raised in this section.

\subsubsection{Fix $p$ and $K$, vary $E$}
We now consider the case when both $p$ and $K$ are fixed and $E$ is allowed to vary.
In this case, we are not able to provide any answers, as we shall explain.
Here, we fix the prime $p\geq 5$ and the imaginary quadratic field $K$ such that $p\nmid h_K$ and $p$ splits in $K$.
As we vary over all $E_{/\Q}$ (base changed to $K$) we need to count how often
\begin{enumerate}
 \item $E$ has good ordinary reduction at $p$.
 \item the primes dividing the conductor (denoted by $N_E$) split in $K$.
 \item $p\nmid \varphi(N)$.
 \item $a_p(E)\not\equiv \pm 1\pmod{p}$.
\end{enumerate}
It is the second condition that makes the set of elliptic curves sparse, and it has not been possible for us to combine this condition with the rest of the conditions.

\newpage
\section{Tables}
\label{section: tables}
\input{sha-table.tex}
\input{heegner-table.tex}
\bibliographystyle{abbrv}
\bibliography{references}
\end{document}

%% file: sha-table.tex
\begin{center}
\begin{table}[h]
\caption{\small{Frequency of prime divisors of $\Sh(E/K^d)$ for varying prime $d$.} }
\label{tab:sha}
\begin{tabular}{|c|c|c|c|c|c|}
\hline
\textbf{} &
  \textbf{\# disc.} &
  \textbf{2} &
  \textbf{3} &
  \textbf{5} &
  \textbf{$\geq$ 7} \\ \hline
{\href{https://www.lmfdb.org/EllipticCurve/Q/11a1/}{\tt 11a1}} & 20 & 4 & 3 & 0 & 0 \\ \hline
\href{https://www.lmfdb.org/EllipticCurve/Q/11a2/}{\tt 11a2} &
  20 &
  4 &
  3 &
  0 &
  0 \\ \hline
\href{https://www.lmfdb.org/EllipticCurve/Q/11a3/}{\tt 11a3} &
  20 &
  4 &
  3 &
  0 &
  0 \\ \hline
\href{https://www.lmfdb.org/EllipticCurve/Q/11a4/}{\tt 11a4} &
  12 &
  2 &
  0 &
  0 &
  0 \\ \hline
\href{https://www.lmfdb.org/EllipticCurve/Q/14a1/}{\tt 14a1} &
  12 &
  2 &
  0 &
  0 &
  0 \\ \hline
\href{https://www.lmfdb.org/EllipticCurve/Q/14a2/}{\tt 14a2} &
  12 &
  7 &
  0 &
  0 &
  0 \\ \hline
\href{https://www.lmfdb.org/EllipticCurve/Q/14a3/}{\tt 14a3} &
  12 &
  2 &
  0 &
  0 &
  0 \\ \hline
\href{https://www.lmfdb.org/EllipticCurve/Q/14a4/}{\tt 14a4} &
  12 &
  2 &
  6 &
  0 &
  0 \\ \hline
\href{https://www.lmfdb.org/EllipticCurve/Q/14a5/}{\tt 14a5} &
  12 &
  7 &
  0 &
  0 &
  0 \\ \hline
\href{https://www.lmfdb.org/EllipticCurve/Q/14a6/}{\tt 14a6} &
  12 &
  7 &
  6 &
  0 &
  0 \\ \hline
\href{https://www.lmfdb.org/EllipticCurve/Q/15a1/}{\tt 15a1} &
  15 &
  7 &
  0 &
  0 &
  0 \\ \hline
\href{https://www.lmfdb.org/EllipticCurve/Q/15a2/}{\tt 15a2} &
  15 &
  7 &
  0 &
  0 &
  0 \\ \hline
\href{https://www.lmfdb.org/EllipticCurve/Q/15a3/}{\tt 15a3} &
  15 &
  11 &
  0 &
  0 &
  0 \\ \hline
\href{https://www.lmfdb.org/EllipticCurve/Q/15a4/}{\tt 15a4} &
  15 &
  7 &
  0 &
  0 &
  0 \\ \hline
\href{https://www.lmfdb.org/EllipticCurve/Q/15a5/}{\tt 15a5} &
  15 &
  11 &
  0 &
  0 &
  0 \\ \hline
\href{https://www.lmfdb.org/EllipticCurve/Q/15a6/}{\tt 15a6} &
  15 &
  7 &
  0 &
  0 &
  0 \\ \hline
\href{https://www.lmfdb.org/EllipticCurve/Q/15a7/}{\tt 15a7} &
  15 &
  7 &
  0 &
  0 &
  0 \\ \hline
\href{https://www.lmfdb.org/EllipticCurve/Q/15a8/}{\tt 15a8} &
  15 &
  11 &
  0 &
  0 &
  0 \\ \hline
  \end{tabular}
  \hspace{3ex}
  \begin{tabular}{|c|c|c|c|c|c|}
\hline
\textbf{} &
  \textbf{\# disc.} &
  \textbf{2} &
  \textbf{3} &
  \textbf{5} &
  \textbf{$\geq$ 7} \\ \hline
\href{https://www.lmfdb.org/EllipticCurve/Q/17a1/}{\tt 17a1} &
  21 &
  12 &
  1 &
  0 &
  0 \\ \hline
\href{https://www.lmfdb.org/EllipticCurve/Q/17a2/}{\tt 17a2} &
  21 &
  12 &
  1 &
  0 &
  0 \\ \hline
\href{https://www.lmfdb.org/EllipticCurve/Q/17a3/}{\tt 17a3} &
  21 &
  12 &
  1 &
  0 &
  0 \\ \hline
\href{https://www.lmfdb.org/EllipticCurve/Q/17a4/}{\tt 17a4} &
  21 &
  12 &
  1 &
  0 &
  0 \\ \hline
\href{https://www.lmfdb.org/EllipticCurve/Q/19a1/}{\tt 19a1} &
  23 &
  6 &
  2 &
  0 &
  0 \\ \hline
\href{https://www.lmfdb.org/EllipticCurve/Q/19a2/}{\tt 19a2} &
  23 &
  6 &
  2 &
  2 &
  0 \\ \hline
\href{https://www.lmfdb.org/EllipticCurve/Q/19a3/}{\tt 19a3} &
  23 &
  6 &
  2 &
  2 &
  0 \\ \hline
\href{https://www.lmfdb.org/EllipticCurve/Q/21a1/}{\tt 21a1} &
  19 &
  11 &
  1 &
  0 &
  0 \\ \hline
\href{https://www.lmfdb.org/EllipticCurve/Q/21a2/}{\tt 21a2} &
  19 &
  11 &
  1 &
  0 &
  0 \\ \hline
\href{https://www.lmfdb.org/EllipticCurve/Q/21a3/}{\tt 21a3} &
  19 &
  15 &
  1 &
  0 &
  0 \\ \hline
\href{https://www.lmfdb.org/EllipticCurve/Q/21a2/}{\tt 21a4} &
  19 &
  11 &
  1 &
  0 &
  0 \\ \hline
\href{https://www.lmfdb.org/EllipticCurve/Q/21a5/}{\tt 21a5} &
  19 &
  15 &
  1 &
  0 &
  0 \\ \hline
\href{https://www.lmfdb.org/EllipticCurve/Q/21a6/}{\tt 21a6} &
  19 &
  11 &
  1 &
  0 &
  0 \\ \hline
\href{https://www.lmfdb.org/EllipticCurve/Q/26a1/}{\tt 26a1} &
  12 &
  1 &
  0 &
  0 &
  0 \\ \hline
  \href{https://www.lmfdb.org/EllipticCurve/Q/26a2/}{\tt 26a2} &
  12 &
  1 &
  0 &
  0 &
  0 \\ \hline
  \href{https://www.lmfdb.org/EllipticCurve/Q/26a3/}{\tt 26a3} &
  12 &
  1 &
  6 &
  0 &
  0 \\ \hline
    \href{https://www.lmfdb.org/EllipticCurve/Q/26b1/}{\tt 26b1} &
  23 &
  9 &
  2 &
  2 &
  0 \\ \hline
    \href{https://www.lmfdb.org/EllipticCurve/Q/26b2/}{\tt 26b2} &
  23 &
  9 &
  2 &
  2 &
  0 \\ \hline
  \end{tabular}
\end{table}
\end{center}

\begin{center}
\begin{table}[h]
\caption{Data for $\mathfrak{S}_p$ for primes $7\leq p<150$.}
\label{tab:2}
\begin{tabular}{ |c|c|c|c| }
\hline
$p$ & $\#\mathfrak{S}_p/p^2$ & $p$ & $\#\mathfrak{S}_p/p^2$ \\
\hline
 7 & 0.0816326530612245 & 71 & 0.0208292005554453 \\
11 & 0.0413223140495868 & 73 & 0.0270219553387127 \\
13 & 0.0710059171597633 & 79 & 0.0374939913475405\\
17 & 0.0276816608996540 & 83 & 0.0178545507330527 \\
19 & 0.0581717451523546 & 89 & 0.0222194167403106 \\
23 & 0.0415879017013233 & 97 & 0.0255074928260176\\
29 & 0.0332936979785969 & 101 & 0.00980296049406921 \\
31 & 0.0312174817898023 & 103 & 0.0288434348194929 \\
37 & 0.0306793279766253 & 107 & 0.00925845051969604 \\
41 & 0.0118976799524093 & 109 & 0.0181802878545577 \\
43 & 0.0567874526771228 & 113 & 0.0263137285613595 \\
47 & 0.0208239022181983 & 127 & 0.0169260338520677 \\
53 & 0.0277678889284443 & 131 & 0.0189382903094225 \\
59 & 0.0166618787704683 & 137 & 0.0108689860940913 \\
61 & 0.0349368449341575 & 139 & 0.0142849748977796 \\
67 & 0.0147026063711294 & 149 & 0.0133327327597856 \\
 \hline
\end{tabular}
\end{table}
\end{center}

%% file: heegner-table.tex
\begin{table}[!ht]
\caption{Some triples $(E,d,p)$ for which $(E,K,p)$ is admissible where $K=\mathbf{Q}(\sqrt{d})$ and $\left(\log_{\omega_E}(P)/p\right)$ is verifiably a $p$-adic unit. $E$ is identified by its Cremona label. To reduce clutter, a blank square is identical to the one above it.}
\label{table:heegner}
  \begin{tabular}{cccc}

  \begin{tabular}[t]{|c|c|c|}
\hline
$E$ & $d$ & $p$ \\ \hline \hline

{\href{https://www.lmfdb.org/EllipticCurve/Q/11a1/}{\tt 11a1}}& -7 & 17  \\ \hline
&     & 19 \\ \hline

&     -19 & 13 \\ \hline
& & 17 \\ \hline
&-35 & 13	\\ \hline
     & & 17	\\ \hline
     & & 19	\\ \hline

     &-39 & 17	\\ \hline
     & & 19	\\ \hline

     &-43 & 13	\\ \hline
     & & 17	\\ \hline

     &-51 & 13	\\ \hline
     & & 19	\\ \hline \hline

 {\href{https://www.lmfdb.org/EllipticCurve/Q/114a1/}{\tt 14a1}}& -31 & 11  \\ \hline
    &     & 13 \\ \hline
    & & 19 \\\hline

    &-47 & 11 \\ \hline
    & & 13 \\ \hline
    & & 17	\\ \hline
         & & 19	\\ \hline
         & -55& 13	\\ \hline

         & & 17	\\ \hline
         & & 19	\\ \hline
       \end{tabular}

       \begin{tabular}[t]{|c|c|c|}
       \hline
       $E$ & $d$ & $p$ \\ \hline \hline
   {\href{https://www.lmfdb.org/EllipticCurve/Q/15a1/}{\tt 15a1}} & -11 & 13 \\ \hline
  &  & 17 \\ \hline
  &  & 19 \\ \hline
  & -59 & 11 \\ \hline
  & & 13\\ \hline
  & & 17\\ \hline
  & & 19\\ \hline \hline

  {\href{https://www.lmfdb.org/EllipticCurve/Q/17a1/}{\tt 17a1}}& -15 & 11  \\ \hline
  &     & 13 \\ \hline
  &     & 19 \\ \hline
  & -19& 5 \\ \hline
  & & 7 \\ \hline
  &     & 11	\\ \hline
     & & 13	\\ \hline
     & -35 & 13	\\ \hline
     && 19	\\ \hline
     &-43 & 7	\\ \hline
     && 11	\\ \hline
     & & 13	\\ \hline
     & & 19	\\ \hline
     & -47 & 7 \\ \hline
     & & 11 \\ \hline
     & & 13 \\ \hline

  \end{tabular}

  \begin{tabular}[t]{|c|c|c|}
  \hline
  $E$ & $d$ & $p$ \\ \hline \hline

     {\href{https://www.lmfdb.org/EllipticCurve/Q/17a1/}{\tt 17a1}}& -47& 19\\ \hline
    & -55& 7	\\ \hline
    & & 13	\\ \hline
    & & 19	\\ \hline
    & -59& 5	\\ \hline
    & & 7	\\ \hline
    & & 11	\\ \hline
    & & 13	\\ \hline
    & & 19	\\ \hline
    & -67& 7	\\ \hline
    & & 11	\\ \hline
    & & 13	\\ \hline
    & & 19	\\ \hline \hline

    {\href{https://www.lmfdb.org/EllipticCurve/Q/19a1/}{\tt 19a1}} & -15 & 11 \\ \hline
    & & 17	\\ \hline
    & -31 & 5	\\ \hline
    & & 7	\\ \hline
    & & 13	\\ \hline
    & & 17	\\ \hline
    & -51& 5	\\ \hline
    & & 7	\\ \hline
    & & 11	\\ \hline
    & & 13	\\ \hline

  \end{tabular}

  \begin{tabular}[t]{|c|c|c|}
  \hline
    $E$ & $d$ & $p$ \\ \hline \hline

  {\href{https://www.lmfdb.org/EllipticCurve/Q/19a1/}{\tt 19a1}}  &-59 & 5	\\ \hline
    & & 7	\\ \hline
    & & 11	\\ \hline
    & & 13	\\ \hline
    & & 17	\\ \hline
    & -67& 7	\\ \hline
    & & 11	\\ \hline
    & & 13	\\ \hline
    & & 17	\\ \hline \hline
    {\href{https://www.lmfdb.org/EllipticCurve/Q/21a1/}{\tt 21a1}} & -47 & 11 \\ \hline
        & & 13	\\ \hline
        & & 17	\\ \hline
      & -59& 5	\\ \hline
          & & 13	\\ \hline
          & & 17	\\ \hline
          & & 19	\\ \hline \hline

         {\href{https://www.lmfdb.org/EllipticCurve/Q/26a1/}{\tt 26a1}} & -23 & 7 \\ \hline
          & & 11	\\ \hline
          & & 17	\\ \hline
          & & 19	\\ \hline
          & -55& 17	\\ \hline
          & & 19	\\ \hline
  \end{tabular}

  \end{tabular}
\end{table}